\documentclass[twoside,a4paper,nopunct]{article}
\usepackage{color}
\usepackage{amssymb}
\usepackage{amssymb}
\usepackage{amssymb}
\usepackage{amssymb}
\usepackage{amsfonts}
\usepackage{amsfonts}
\usepackage{amsfonts}
\usepackage{amsfonts}
\usepackage{mathrsfs}
\usepackage{graphicx}
\usepackage{amsmath}
\usepackage{amssymb}
\usepackage{array}
\usepackage{float}
\usepackage{makecell}
\usepackage{cite}
\usepackage{tabu}
\usepackage{cite}
\usepackage{multirow}
\usepackage{longtable}
\usepackage{graphicx}
\usepackage{float}
\usepackage[justification=centering]{caption}
\numberwithin{equation}{section}
\newtheorem{theorem}{Theorem}[section]
\newtheorem{lemma}[theorem]{Lemma}
\newtheorem{cor}[theorem]{Corollary}

\newenvironment{proof}[1][Proof]{\textbf{#1.} }

\date{}

\begin{document}

\setcounter{page}{1}

\title{A New Characterization of Sporadic Groups \thanks{This work was supported by the National Natural Science Foundation of China
		(Grant No.11671324, 11971391), Fundamental Research Funds for the Central Universities (No. XDJK2019B030).}}


\author
{ \small $\mbox{Zhongbi Wang}$, $\mbox{Heng Lv}$, $\mbox{Yanxiong Yan},$ $\mbox{Guiyun Chen}$\thanks{corresponding author}      \\
\small  School of Mathematics and Statistics, Southwest University,\\
\small Beibei, Chongqing, 400715,  China\\
\small  E-mail: zbwango1@163.com; lvh529@swu.edu.cn;  2003yyx@163.com;  gychen@swu.edu.cn}

\maketitle
\begin{quote}

{\small {\bf Abstract: Let $G$ be a finite group,  $n$ a positive integer.  $\pi(n)$ denotes  the set of all prime divisors of $n$ and $\pi(G)=\pi(|G|)$. The prime graph $\Gamma(G)$ of $G$, defined by Grenberg and Kegel,  is a graph whose vertex set is $\pi(G)$, two vertices $p,\ q$ in $\pi(G)$ joined by an edge if and only if $G$ contains an element of order $pq$. In this article, a new characterization of sporadic simple groups is obtained, that is, if $G$ is  a finite group and  $S$ a sporadic simple group. Then $G\cong S$ if and only if $|G|=|S|$ and  $\Gamma(G)$ is disconnected. This characterization unifies the several characterizations that can conclude the group has non-connected prime graphs, hence several known characterizations of sporadic simple groups become the corollaries of this new characterization.
}\\

}\\
\end{quote}

\baselineskip 22pt
\section{Introduction}

 ~~~~~~~In the past three decades, as a very interesting topic, quantitative characterization of a group, especially a simple group, has been being an active topic in the theory of finite simple group since classification of finite simple groups completed  in the early of 1980s. When Wujie Shi began to investigate the topic whether a finite simple group can be uniquely determined by its order and the set of its element orders, he proposed a famous conjecture in 1987, which was recorded as Problem 12.39 in \cite{Unsolved}.

\textbf{Shi's Conjecture.} Let $G$ be a finite group, $S$ a finite simple group, then $G\cong S$ if and only if $|G|=|S|$ and $\pi_e(G)=\pi_e(S)$, where $\pi_e(G)$ denotes the
set of element orders in $G$.

Research on Shi's conjecture began an era of quantitative  characterization of finite simple groups. At last, this  conjecture was completely proved in 2009. In the series of papers to prove Shi's conjecture, an important concept the prime graph of a finite group was frequently used for dealing with those simple groups with non-connected prime graph, which was defined  by Gruenberg and Kegel in \cite{Williams} as following:

  Let $G$ be a finite group,  $n$ a positive integer.  $\pi(n)$ denotes  the set of all prime divisors of $n$ and $\pi(G)=\pi(|G|)$. The prime graph $\Gamma(G)$ of $G$, defined by Grenberg and Kegel,  is a graph whose vertex set is $\pi(G)$, two vertices $p,\ q$ in $\pi(G)$ joined by an edge if and only if $G$ contains an element of order $pq$. We denote the number of connected components of $\Gamma(G)$ as $t(G)$, the connected components of $\Gamma(G)$ as $\{\pi_i,\ i=1,\cdots,\ t(G)\}$, and we always assume $2$ is in $\pi_1(G)$ if $2\big||G|$. The components of prime graphs of all finite simple groups were given by J. S. William, A. S. Kondrat$\acute{e}$v, M. Suzuki, N. Iiyora and H. Yamaki etc. (see [2-5]).

The prime graph once be used by the second author to study the famous Thompson's conjecture:

\textbf{Thompson's Conjecture.} Let $G$ be a finite group with $Z(G)=1$, $N(G)=\{n\in \mathbf{N}| G \mbox{ has a conjugacy } $ $\mbox{ class of length } n\}.$ If $M$ is  a finite simple group such that $N(G)=N(M)$, then $G\cong M.$

During the second author studying Thompson's conjecture, he proved that $|G|=|M|$ if $G$ and $M$ satisfy conditions of Thompson's Conjecture and the prime graph $\Gamma(M)$ is non-connected. For a finite group with non-connected prime graph, he divided its order into co-prime divisors, each of them exactly corresponding to each of components of the prime graph, and called these divisors the order components of the finite group and found that many finite simple groups can  be uniquely determined by  their order components. Actually, many simple groups with non-connected prime graphs have been proved to be uniquely determined by their order components, for example, in a series of articles , for example [6-11], etc,  it is shown  many simple groups with non-connected prime graphs are characterized by order components of their prime graphs. There are some other topics on characterization of a finite simple group by its order and some other quantitative properties related to non-connected prime graph, for example, in [12-16], characterization of a finite simple group by its order and maximal element order (the largest element order or the second largest element order, or both of them), characterization of a finite simple group by its order and the set of orders of maximal abelian subgroups, etc. In these topics, the discussed finite group are usually ascribed to a finite group having non-connected prime graph, whose order is the same as some finite simple group. Therefore, it is a meaningful topic to study the finite group having its prime graph non-connected and its order being equal to a finite simple group. In this article, we shall discuss this topic and specially focus on a finite group having its prime graph non-connected and its order being equal to a sporadic group. we shall prove the following theorem:

\textbf{Main Theorem}. \textit{ Let $G$ be a finite group and  $S$ a sporadic simple group. Then $G\cong S$ if and only if $|G|=|S|$ and the prime graph of $G$ is disconnected.}

By above theorem, the following known characterizations of sporadic simple groups， including Shi' Conjecture and Thompson Conjecture, become its corollaries since under the hypothesis the prime graphs of the groups are non-connected, for example:

\begin{cor}
	Let $G$ be a finite group and  $S$ a sporadic simple group.   Then $G\cong S$ if and only if $|G|=|S|$ and $\pi_e(G)=\pi_e(S)$.
\end{cor}
\begin{cor}
Let $G$ be a finite group and  $S$ a sporadic simple group.   Then $G\cong S$ if and only if $Z(G)=1$ and $N(G)=N(S)$.
\end{cor}
\begin{cor}
	Let $G$ be a finite group and  $S$ a sporadic simple group.   Then $G\cong S$ if and only if $G$ and $S$ have the same order components.
\end{cor}
\begin{cor}
Let $G$ be a finite group and  $S$ a sporadic simple group.  Then $G\cong S$ if and only if $|G|=|S|$ and the sets of orders of maximal abelian subgroups of $G$ and $S$ are equal.
\end{cor}

\begin{cor}
Let $G$ be a finite group and  $S$ a sporadic simple group.   Then $G\cong S$ if and only if $|G|=|S|$ and the largest element orders of $G$ and $S$ are the same.
\end{cor}

Throughout the paper, the actions unspecified always means conjugate action.  Let $G$ be a finite group, for a prime $p\in \pi(G)$, $G_p$ denotes the $Sylow$ $p-$subgroup of $G$. In addition, $G=U\ltimes V$ denotes $G$ is the semidirect product of $U$ and $V$, especially, $V\unlhd G$.

\section{ The Proof of the Main Theorem}

In order to prove the Main Theorem, we first present some lemmas which are useful in our proof.
\begin{lemma}\label{ty1}(\cite{Williams}, Theorem $A$)~ If $G$ is a finite group whose prime graph has more than one components, then one of the following holds:

(1) $G$ is a $Frobenius$ group;

(2) $G$ is a  $2-$$Frobenius$ group;

(3) $G$ has a normal series $1\unlhd H \unlhd K \unlhd G $ such that $H$ and $G/K$ are $\pi_1$ groups, $K/H$ is a non-abelian simple group, $H$ is a nilpotent group, where $2\in \pi_1$. And $|G/K|\big| |\mathrm{Out}(K/H)|$.

\end{lemma}

\begin{lemma}\label{ty2}(\cite{Italian}, Lemma 2.6(1)) ~Let $G$ be a $Frobenius$ group with $Frobenius$ kernel $K$ and  $Frobenius$ complement $H$. Then $K$ is nilpotent and $|H|\big||K|-1$. Moreover, $ t(G)=2$, and  $\Gamma(G)=\{\pi(H),\ \pi(K)\}.$
\end{lemma}

\begin{lemma}\label{ty3}(\cite{Italian}, Lemma 2.6(2))~
 Let $G$ be a $2-$$Frobenius$ group, then $G=DEF,$ where $D$ and $DE$ are normal subgroups in $G$, $DE$ and $EF$ are $Frobenius$ groups with kernels $D$ and $E$, respectively. Moreover $t(G)=2,$ $\pi_1(G)=\pi (D)\cup \pi (F)$ and $\pi_2(G)=\pi(E)$.
\end{lemma}

\begin{lemma}\label{ty4}(\cite{Kurzweil}, 8.2.3)~ Let $G$ and $A$ be two groups. Let $p$ be a prime divisor of $|G|$. Suppose that the action of $A$ on $G$ is $coprime$. Then there exists an $A-$invariant $Sylow \ p-$subgroup of G.
\end{lemma}

\begin{lemma}\label{ty5}(\cite{Italian}, Lemma 2.4)~ Let $G$ be a $p-$group of order $p^n,$ and $G$ act on a $q-$group $H$ of order $q^m$, where $p,\ q$ are two distinct primes. If $|G|\nmid |\rm{GL}(m,q)|,$ then $pq\in \pi_e(G\ltimes H)$.\end{lemma}

By the above Lemma, we have the following corollary.
\begin{cor}\label{cor51}Let $G$ be a $p-$group of order $p^n,$ and $G$ act on a $q-$group $H$ of order $q^m$. If  $p^r\nmid\prod_{i=1}^m(q^i-1)$, where $1\leq r\leq n$, then $pq\in \pi_e(G\ltimes H)$. \end{cor}

For convenience, we call a section of a group $G$ that is a $Frobenius$ group as a section of $Frobenius$ type.

\begin{lemma}\label{ty61}Let $G$ be a finite group, $p, q\in \pi(G)$. If $|G_q|=q^n$, and $p\nmid |\rm{GL}(n,q)|,$ then $G$ has no section of $Frobenius $ type $H\ltimes K$ such that  $p\big||H|$ and $|K_q|=q^r$, where $r\leq n$. Furthermore, if $G$ does have a section of  $Frobenius $ type $H\ltimes K$ such that $pq\big||HK|$ and $p\big||H|$, then $(q, |K|)=1$ and $pq\big||H|$.\end{lemma}
\begin{proof} If $G$ has a section of $Frobenius $ type $H\ltimes K$ such that $p\big||H|$ and $|K_q|=q^r$,  then $H_p\ltimes K_q$ is a $Frobenius$ group and  $H_p\ltimes K_q$  has no element of order $pq$. But under the assumption $p\nmid |\rm{GL}(n,q)|,$ we have $|H_p|\nmid|\rm{GL}(n,q)|$, since $|\rm{GL}(r, q)|\big||\rm{GL}(n,q)|$, it follows that $|H_p|\nmid |\rm{GL}(r,q)|$, thus, it follows by Lemma \ref{ty5} that $H_p\ltimes K_q$ has an element of order $pq$, a contradiction. The first part of the lemma follows. And the second part follows straightforward from the first part. \end{proof}

\textbf{Remark.} Let $G$ be a finite group satisfying the hypothesis of the Main Theorem, then by Lemma \ref{ty1}, $G$ may be a $Frobenius$ group or a $2-Frobenius$ group. In order to prove the Main Theorem, we first show that $G$ cannot be  a $Frobenius$ group or a $2-Frobenius$ group, the proof will be separated into six lemmas. In the proofs of next six lemmas, when we mention  a $Frobenius$ group or a $2-Frobenius$ group, we are referring groups and notations in Lemma \ref{ty2} and Lemma \ref{ty3} without explanations. Moreover, we shall frequently use Lemma \ref{ty61} to get information of $\pi(H)$ or come to a contradiction.

\subsection{To prove that $G$ is neither a $Frobenius$ group nor a $2-$$Frobenius$ group}
Now we start to show that $G$ is neither a $Frobenius$ group nor a $2-$$Frobenius$ group.
\begin{lemma}\label{ty6}Let $G$ be a group and $S$ a Mathieu group. If $|G|=|S|$, then $G$ is neither a $Frobenius$ group nor a $2-$$Frobenius$ group. \end{lemma}
\begin{proof}  From  \cite{ATLAS}, we get  the orders of Mathieu groups are the following:

$|M_{11}|=2^4\cdot 3^2\cdot5\cdot11$,
$|M_{12}|=2^6\cdot 3^3\cdot5\cdot11$,
$|M_{22}|=2^7\cdot 3^2\cdot5\cdot7\cdot11$,
$|M_{23}|=2^7\cdot 3^2\cdot5\cdot7\cdot11\cdot23$,
$|M_{24}|=2^{10}\cdot 3^3\cdot5\cdot7\cdot11\cdot23$.

Let $U\ltimes V$ be a $Hall$ subgroup of $G$ with  $11\big||UV|$ or $23\big||UV|$, which is a $Frobenius$ group with $V$ the $Frobenius$ kernel and $U$ the $Frobenius$ complement.
  Let $S$ be one of $M_{11}$, $M_{12}$ and $M_{22}$. Since $11\nmid |\rm{GL}(7, 2)|$, $11\nmid |\rm{GL}(3, 3)|$ and $5,7<11$, we have by Lemma \ref{ty61} that $11\in \pi(V)$. Thus, $|V_{11}|=11$. Since $V_{11}U$ is a $Frobenius$ group, we deduce that $|U|\big|10$. Noticing $2^4\big||G|$, we conclude that $|U|=5$. If $G$ is a $2-Frobenius$ group, then  $11\in\pi(D)$ for both $E$ and $F$ are two $Frobenius$ complements. Surely,  $|D_{11}|=11$. Thus, $|E|=5$, and hence, $|F|\big|4$. Therefore, $3^2\big||D|$, it follows that $15\in\pi_e(DE)$, a contradiction. If $G$ is a $Frobenius$ group, then $G=UV$ and  $3^2\big||V|$, but $5\nmid |\rm{GL(2,3)}|$, so $G$ has an element of order $15$ by Lemma \ref{ty5}, a contradiction.

 Let $S$ be one of $M_{23}$ or $M_{24}$. We consider the prime $23$. Since $23\nmid |\rm{GL}(10,2)|$, $23\nmid |\rm{GL}(3,3)|$, and $5,7,11<23$, in the same reason we have that $23\in\pi(V)$. Now $|V_{23}|=23$,  it follows that $|U|=11$ by Corollary \ref{cor51}. Suppose that  $G$ is a $2-Frobenius$ group, then  $23\in\pi(D)$, thus,  $|E|=11$, therefore, $|F|\big||10|$. Consequently, $7\in\pi(D)$ and $|D_7|=7$, which is impossible since $D_7E$ cannot be a  $Frobenius$ group. If $G$ is a $Frobenius$ group,  then $G=UV$. Hence, $UV_7$ is a $Frobenius$ group, contradicting to Lemma \ref{ty61}. This concludes the lemma.

  \end{proof}

\begin{lemma} \label{ty7} Let $G$ be a group and $S$ a Janko group. If $|G|=|S|$, then $G$ is neither a $Frobenius$ group nor a $2-$$Frobenius$ group.
\end{lemma}

\begin{proof} \textbf{Step 1} To show the lemma follows while $|G|=|J_1|$ or $|J_3|$.

Noticing $|J_1|=2^3\cdot3\cdot5\cdot7\cdot11\cdot19$,   $|J_3|=2^7\cdot3^5\cdot5\cdot17\cdot19$, $19\nmid |GL(7,2)|$ and $19\nmid |GL(5,3)|$, we get by Lemma \ref{ty61} that $G$ has no subgroups which is a $Frobenius$ group with complement of order divided by $19$  if $|G|=|J_1|$ or $|J_3|$.  Hence, if $G$ is a $2-Frobenius$ group, then $19\in\pi(D)$, $|E|\big|9$ for $|E|$ is odd, further, $|F|\big|8$ since $|F|\big|(|E|-1)$. Therefore, $|D_5|=5$, which means that $ED_5$ is a $Frobenius$ group of order $15$ or $45$, it is impossible. If $G$ is a $Frobenius$ group,  then  $|H|\big|9$, $|K_5|=5$. Thus, $H_3K_5$ is a $Frobenius$ group,  which is impossible.

\textbf{Step 2} To show the lemma follows while $|G|=|J_2|$.

(1) Let  $U\ltimes V$ be a section of $Frobenius$ type of $G$. Notice $|G|=|J_2|=2^7\cdot3^3\cdot5^2\cdot7$. If $7\in \pi(V)$, then $UV_7$ is $Frobenius$ group, thus, $|U|\big| 6$. If $7\in\pi(U)$, by $7\nmid\rm{GL}(3,3)|$ and $7\nmid|\rm{GL}(2,5)|$, we get by Lemma \ref{ty61} that $V$ is $2-$group.

 (2) Assume that $G$ is a $Frobenius$ group. If $7\in\pi(K)$, then $|H|\big|6$ by (1), hence, $\pi(K)=\pi(G)$ and $t(G)=1$, a contradiction. If $7\in\pi(H)$, then $|K|=2^7$, $|H|=3^3\cdot5^2\cdot7$, contradicting  $|H|\big|(|K|-1)$.

  (3) Assume that $G$ is a $2-Frobenius$ group. If $7\big||D|$, then $|E|\big|6$ by (1), so $|F|\big|2$, thus, $\pi(D)=\pi(G)$ and $t(G)=1$, a contradiction. If $7\in\pi(E)$, then $D$ is a $2-$group by (1), so  $|F|\big|6$. Since $3^3\big||G|$,  we deduce that $|E|=3^3\cdot5^2\cdot7$, contradicting $|E|\big|(|D|-1)$. As a result, $7\in\pi(F)$, since $EF$ is a $Frobenius$ group,  (1) implies  that $E$ is a $2-$ group, a contradiction.

\textbf{Step 3} To show the lemma follows while   $|G|=|J_4|$.

(1) By $|G|=|J_4|=2^{21}\cdot3^3\cdot5\cdot7\cdot11^3\cdot23\cdot29\cdot31\cdot37\cdot43$, we get that $G$ has no nilpotent normal subgroup of order divided by $5$, otherwise, $t(G)=1$. Moreover, by $43\nmid\rm{GL}(3,3)|$ and $43\nmid\rm{GL}(3,11)|$, we get that if $U\ltimes V$ is a section of $Frobenius$ type with  $43\big||U|$, then $V$ is a $2-$group.

 (2) Assume that $G$ is a $Frobenius$ group. Then  $5\in\pi(H)$.   Moreover, if $43\in\pi(K),$ then Corollary \ref{cor51} indicates that $43\cdot5\in\pi_e(G)$, this is impossible. Hence, $43\in\pi(H)$, so $|K|=2^{21}$ by (1), thus,  $|H|=3^3\cdot5\cdot7\cdot11^3\cdot23\cdot29\cdot31\cdot37\cdot43$, contradicting $|H|\big|(|K|-1)$.

 (3) Assume that $G$ is a $2-Frobenius$ group, then  $5\notin\pi(D)$ by (1). Moreover,   if  $43\in\pi(D)$, then $\pi(E)\subseteq\{3, 7\}$, and thus, $5\in\pi(F)$. But by Lemma \ref{ty61}, $EF_5$ cannot be a $Frobenius$ group, a contradiction. If $43\in\pi(E)$, then  $D$ can only  be a $2-$group by (1). At this moment, $F$ is a $\{3, 7\}-$group for $E_{43}F$ is a $Frobenius$ group. Hence, $5\in\pi(E)$ and $FE_5$ is a $Frobenius$ group, which is impossible for $\pi(F)\subseteq\{3,7\}$, a contradiction. Now $43\in\pi(F)$, (1) indicates that $E$ is a $2-$group, a contradiction.

 The lemma follows from Steps 1-3.
\end{proof}

\begin{lemma} \label{ty8} Let $G$ be a group and $S$ a Conway group. If $|G|=|S|$, then $G$ is neither a $Frobenius$ group nor a $2-$$Frobenius$ group.
\end{lemma}

\begin{proof}  From  \cite{ATLAS}, we get  the orders of Conway groups are the following:
 $$|Co_{1}|=2^{21}\cdot3^9\cdot5^4\cdot7^2\cdot11\cdot13\cdot23,
\ |Co_{2}|=2^{18}\cdot3^6\cdot5^3\cdot7\cdot11\cdot23,\
 |Co_3|=2^{10}\cdot3^7\cdot5^3\cdot7\cdot11\cdot23.$$

(1) Now $2^{10}\cdot 3^6\cdot5^3$ divides $|G|$. Noticing  $|\rm{GL}(2,7)|=2^5\cdot3^2\cdot7$, we have that if $G$ has a nilpotent normal subgroup of order divided by $7$ or $11$, then $t(G)=1$ by Corollary \ref{cor51}. Hence, $G_7$ and $G_{11}$ are not normal in $G$.  Let $U\ltimes V$ be a section of $Frobenius$ type, then, by $23\nmid |\rm{GL}(9,3)|$,  $23\nmid |\rm{GL}(4,5)|$,  $23\nmid|\rm{GL}(2,7)|$ and $11,13<23$, we get that if $23\big||U|$ then $V$ is a $2-$group.

 (2) If $G$ is a $Frobenius$ group. Suppose that $23\in\pi(K)$, then $|H|=11$, so that $7\in\pi(K)$, contradicting (1). Consequently,  $23\in\pi(H)$ and $|K|=|G_2|$ by (1), a contradiction to $|H|\big|(|K|-1)$. Let $G$ be a $2-$$Frobenius$ group. Since $|E|$ is odd, we have $23\nmid|F|$ by (1).  If $23\in\pi(E)$, then $D$ is a $2-$group and $|F|=11$. Therefore, $7\in\pi(E)$, which yields $|F|$ is a $\{2,3\}-$group, a contradiction to $|F|=11$.  Hence, $23\in\pi(D)$, we deduce that $|E|=11$ and $|F|\big|10$. Thus, $7\in\pi(D)$, contradicting (1). This  concludes the lemma.

\end{proof}

\begin{lemma} \label{ty9} Let $G$ be a group and $S$ a Fischer group. If $|G|=|S|$, then $G$ is neither a $Frobenius$ group nor a $2-$$Frobenius$ group. \end{lemma}

\begin{proof} From  \cite{ATLAS}, we get  the orders of Fischer groups are the following:
\begin{center}
 $|Fi_{22}|=2^{17}\cdot3^9\cdot5^2\cdot7\cdot11\cdot13$, $|Fi_{23}|=2^{18}\cdot3^{13}\cdot5^2\cdot7\cdot11\cdot13\cdot17\cdot23$, $|Fi_{24}^{\prime}|=2^{21}\cdot3^{16}\cdot5^2\cdot7^3\cdot11\cdot13\cdot17\cdot23\cdot29$.
 \end{center}

 (1) Similar to (1) of the proof of Lemma \ref{ty8}, and note that $|\rm{GL(3,7)}|=2^6\cdot3^4\cdot7^3\cdot19$, we can show that $G_p$ (if there is) is not normal in $G$, where $p=5, 7, 11, 13, 17$.  Let $U\ltimes V$ be a section of $Frobenius$ type of $G$, then by $\pi(\rm{GL}(2,5))=\{2,3,5\}$ and $\pi(\rm{GL}(3,7))=\{2,3,7,19\}$ and Corollary \ref{cor51}, we get that if $13\big||U|$, then $V$ is a $\{2,3\}-$group.

(2) Assume that $G$ is a $Frobenius$ group. If $|G|=|Fi_{22}|$, we have $5,7,11,13\in\pi(H)$, and $K$ is a $2-$group,  contradicting $|H|\big|(|K|-1)$. If $|G|=|Fi_{23}|$ or $|Fi^{\prime}_{24}|$, then we come to a contradiction by (1) and $13$ does not divide  $2^{18}-1$, $2^{21}-1$, $3^{13}-1$ and $3^{16}-1$.

 (3) Assume that $G$ is a $2-Frobenius$ group.
Then either $13\in\pi(E)$ or $13\in\pi(F)$ by (1). Suppose that $13\in\pi(E),$ since $E_{13}F$ is a $Frobenius$ group, it follows that $|F|\big|2^2\cdot3$, so that $11\in\pi(E)$ by (1), thus, $|F|=2$. Hence, $|D_2|=|G_2|/2$. Note that $D$ is a $\{2,3\}-$group, we have  $|E_5|=|G_5|=5^2$. Furthermore, by $5^2\nmid3^t-1,\ t=9,13,16 $, we deduce that $3\nmid|D|$.  Therefore, $D$ is a $2-$group. And since $DE_3$ is a $Frobenius$ group, we immediately  get a contradiction by $3^9\nmid 2^k-1$, where $k=16,17,20$.  As a result, $13\in\pi(F)$, then $E$ can only be a $3-$group. Which implies that $D$ is a $2-$group. By calculating $2^m-1$, $m\leq 21$, we get that $3^9\nmid 2^m-1$, a contradiction. This completes the proof of the lemma.

\end{proof}

\begin{lemma} \label{ty10} Let $G$ be a group and $S$ the  Monster group or the Baby group. If $|G|=|S|$, then $G$ is neither a $Frobenius$ group nor a $2-$$Frobenius$ group. \end{lemma}

\begin{proof} From  \cite{ATLAS}, the orders of the Monster group and the Baby group are the following:

$|M|$= $2^{46}\cdot3^{20}\cdot5^9\cdot7^6\cdot11^2\cdot13^3\cdot17\cdot19\cdot23\cdot29\cdot31\cdot41\cdot47\cdot59\cdot71$,
$|B|$= $2^{41}\cdot3^{13}\cdot5^6\cdot7^2\cdot11\cdot13\cdot17\cdot19\cdot23\cdot31\cdot47$.

(1) It can be easily shown that $t(G)=1$ if $G_{p}$ is normal in $G$, where $p=17,19,29, 31,41$ or $71$, hence, $G$ has no  normal subgroup with order $p$. Moreover, by calculating, it follows that $47$ does not divide $|\rm{GL}(20,3)|$, $|\rm{GL}(9,5)|$, $|\rm{GL}(6,7)|$, $|\rm{GL}(2,11)|$ and  $|\rm{GL}(3,13)|$. Hence, if $G$ has a section of $Frobenius$ type $U\ltimes V$ with  $47\big||U|$, then $V$ can only  be a $2-$group.

 (2) Assume  that $G$ is a $Frobenius$ group. If  $47\in\pi(K)$, then $|K_{47}|=47$ and $K_{47}\unlhd G$. So that $|H|=23$, which yields $17\in\pi(K)$, a contradiction. Therefore, $47\in\pi(H)$, which implies $|K|=|G_2|$ by (1), contradicting  $|H|\big|(|K|-1)$.

 (3) Assume that $G$ is a $2-Frobenius$ group. Then $17\not\in\pi(D)$ by (1). If  $47\in\pi(D)$, then $|E|=23$, hence, $|F|\big|2\cdot11$, from which it follows that $17\in\pi(D)$, a contradiction to (1). If  $47\in\pi(E)$, then $|F|$ divides $2\cdot 23$, and $D$ is a $2-$group by (1). Consequently, for any $p\in\pi(G)\setminus\{2,23\}$,  all $p-$elements are contained in $E$, so $|E|\geq 3^{13}\cdot5^6\cdot7^2\cdot11\cdot13\cdot17\cdot19\cdot31\cdot47>2^{46}$,  contradicting $|E|\big|(|D|-1)$. As a result, $47\in\pi(F)$, then $E$ can only be a $2-$group by (1) since $EF_{47}$ is a $Frobenius$ group, a contradiction to $|E|$ is  odd.  The Lemma follows.\end{proof}

\begin{lemma} \label{ty11} Let $G$ be a group and $S$ one of $Suz,\ HS,\ M^{c}L,\ He,\ HN,\ Th,\ O^{\prime}N,$ $Ly$ and $Ru$. If $|G|=|S|$,  then $G$ is neither a $Frobenius$ group nor a $2-$$Frobenius$ group. \end{lemma}

\begin{proof} We write the proof in five steps.

\textbf{Step 1} \ To show the lemma follows while $|G|$ equals one of  $|Suz|$, $|HS|$ and  $|M^cL|$.

From  \cite{ATLAS}, we get that:
 $|Suz|=2^{13}\cdot3^7\cdot5^2\cdot7\cdot11\cdot13$;
$|HS|=2^{9}\cdot3^2\cdot5^3\cdot7\cdot11$; $|M^cL|=2^{7}\cdot3^6\cdot5^3\cdot7\cdot11$.

(1) If $G_{11}\unlhd G$, then by $|G_{11}|=11$ we can show that $t(G)=1$, a contradiction. Hence, $G_{11}\ntrianglelefteq G$. Similarly, we get that $G_7\ntrianglelefteq G$. Let $U\ltimes V$ be a section of $Frobenius$ type of $G$ such that $11\big||U|$. Since  $11\nmid |\rm{GL}(3,5)|$, then $V$ can only  be a $\{2,3\}-$group.

(2) Assume  that $G$ is a $Frobenius$ group. Then  $11\in\pi(H)$ by (1), so that $K$ can only be a $\{2,3\}-$group. Note that $|H|\big|(|K_p|-1)$ for $p=2$ or $3$, it is impossible by trivial calculating.

(3) Assume that $G$ is a $2-Frobenius$ group. Then  $7,11\not\in\pi(D)$  by (1).  Suppose that $11\in\pi(E)$, then  $|F|\big|2\cdot5$ and $D$ is a $\{2,3\}-$group by (1). This indicates that  $7\in\pi(E)$.  Hence, $|F|=2$, thereby, $5\in\pi(E)$. If $3\in\pi(D)$, then, in view of the facts that $5\parallel|GL(7,3)|$ and $5^2\big||G|$, there exists an element of order 5 in $E$ commutes an element of order 3, a contradiction. So $3\in\pi(E)$. Therefore, $D$ is a $2-$subgroup, and thus, all the $2^{\prime}-$elements are  in $E$, it is impossible for $|E|\big|(|D|-1)$ and $|D|=|G_2|/2$. Now $11\in\pi(F)$, then $E$ is  a $3-$group by (1) and $|E|$ odd. In this case $E=G_3$, but by checking we know that $11\nmid |G_3|-1$, a contradiction to $EF_{11}$ is  a $Frobenius$ group.

\textbf{Step 2}\  To show the lemma follows while $|G|=|He|$.

(1) By $|G|=|He|=2^{10}\cdot3^3\cdot5^2\cdot7^3\cdot17$, one has  that $|G_{17}|=17$, hence, $G$ has no nilpotent normal subgroup with order divided by $17$ for $t(G)>1$. Moreover, by $17$ does not divide the orders of $GL(3,3)$, $GL(2,5)$ and $GL(3,7)$, we  get that  if $G$ has a section of $Frobenius $ type $U\ltimes V$ with  $17\big||U|$ then $V$ is a $2-$group.

 (2) Assume  that $G$ is a $Frobenius$ group. Then  by (1) it follows that $17\in\pi(H)$, and $|K|=2^{10}$. Hence, $|H|=3^3\cdot5^2\cdot7^3\cdot17$, contradicting $|H|\big|(|K|-1)$. Assume that $G$ is a $2-Frobenius$ group.  Then we get that $17\in\pi(E)$ or $17\in\pi(F)$ by (1).  Suppose that $17\in\pi(E)$, then   $D$ and $F$ are two   $2-$groups by (1) and $|F|\big|16|$.   As a result, $\pi(E)=\{3,5,7,17\}$. Contradicting  $|E|\big|(|D|-1)$.  Hence,   $17\in\pi(F)$, yielding that  $E$ is a $2-$group by (1), a contradiction to $|E|$ is  odd.

\textbf{Step 3} To show the lemma follows while $|G|$ equals one of $|HN|,\ |Th|$ and $ |O^{\prime}N|$.

(1) By \cite{ATLAS}, $|HN|=2^{14}\cdot3^6\cdot5^6\cdot7\cdot11\cdot19$;
  $|Th|=2^{15}\cdot3^{10}\cdot5^3\cdot7^2\cdot13\cdot19\cdot31$;
  $|O^{\prime}N|=2^{9}\cdot3^{4}\cdot5\cdot7^3\cdot11\cdot19\cdot31$.
Hence, $|G_{19}|=19$, and $G_{19}$ is not normal in $G$, otherwise, $t(G)=1$. Moreover, since $19$ does not divide $|\rm{GL}(15,2)|$, $|\rm{GL}(10,3)|$, $|\rm{GL}(6,5)|$ and  $|\rm{GL}(2,7)|$, we  get that   if $G$ has a section of $Frobenius $ type $U\ltimes V$ with  $19\big||U|$, then $|V|=7^3$.

(2) Assume  that $G$ is a $Frobenius$ group, then by (1) it follows that $19\big||H|$, $K=G_7$ and $|K|\big|7^3<|H|$, contradicting $|H|\big|(|K|-1)$.

 (3) Assume that $G$ is a $2-Frobenius$ group. Then $19\in\pi(E)$ or $19\in\pi(F)$ by (1).  If $19\in\pi(E)$, since $DE_{19}$ is a $Frobenius$ group, we have $D$ is a $7-$ group, further, by $19\big||D|-1$, we have $|D|=7^3$. Therefore, $|G|=|O'N|$ and $|F|=2$ for $|F|\big||E_{19}|-1$. As a result,  $|E|> |D|-1$, a contradiction. If $19\in\pi(F)$, then $|E|=7^3$ by (1) and previous argument, so $|G|=|O'N|$. Noticing that $7^3-1=19\cdot2\cdot3^2$. Therefore, $|D_{11}|=11$, which is impossible for $ED_{11}$ is a $Frobenius$ group.

\textbf{Step 4} To show the lemma follows while $|G|=|Ly|$.

 (1) In such case, $|G|=|Ly|=2^{8}\cdot3^{7}\cdot5^6\cdot7\cdot11\cdot31\cdot37\cdot67$.
Since $67$ does not divide $|\rm{GL}(8,2)|$, $|\rm{GL}(7,3)|$ and  $|\rm{GL}(6,5)|$, we come to that $G$ has no  section of   $Frobenius $ type $U\ltimes V$ with $67\big||U|$.

 (2) Assume  that $G$ is a $Frobenius$ group.  Then  $67\in\pi(K)$ by (1), hence,  $|H|\big|66$. Therefore, $7\in\pi(K)$ and $|K_7|=7$, we get a contradiction for $K_7H$ is a $Frobenius$ group.

 (3) Assume that $G$ is a $2-Frobenius$ group.  Then  $67\in\pi(D)$ by (1) and $E$,  $F$ are two $Frobenius$ complements.
 Hence, $|E|\big| 33$, which concludes $|F|\big|2\cdot5$. Thereby, $7\in\pi(D)$, since $|D_7|=7$, we get a contradiction for $D_7E$ is a $Frobenius$ group.

\textbf{Step 5} To show the lemma follows while $|G|=|Ru|$.

(1) In this case, $|G|=|Ru|=2^{14}\cdot3^{3}\cdot5^3\cdot7\cdot13\cdot29$. Since $29$ does not divide $|\rm{GL}(14,2)|$, $|\rm{GL}(3,3)|$ and  $|\rm{GL}(3,5)|$, we come to that $G$ has no  section of  $Frobenius $ type $U\ltimes V$ with $29\big||U|$. Noticing that $|G_{13}|=13$, we deduce that $G$ has no nilpotent normal subgroup with order divided by $13$.

 (2) Assume  that  $G$ is a $Frobenius$ group. Then,  $29\in\pi(K)$ by (1), thus,  $|H|\big| 28$. Hence, $13\in\pi(K)$, a contradiction to (1).  Assume that $G$ is a $2-Frobenius$ group. Then,  $29\in\pi(D)$. And thus,  $|E|=7$ for $|E|$ odd, it follows that $|F|\big|2\cdot3$. Thereby, $13\in\pi(D)$, a contradiction to (1).

 This Lemma follows from Step 1-5.
\end{proof}

\begin{cor}\label{ty12}Let $G$ be a finite group and  $S$  one of 26 sporadic groups.   If $|G|=|S|$ and the prime graph of $G$ is disconnected,  then $G$ has a normal series $1\unlhd H \unlhd K \unlhd G $ such that the following property (*):\\	
(*)\ \ \ \ $H$ and $G/K$ are $\pi_1$ groups, $K/H$ is a non-abelian simple group, $H$ nilpotent, and $|G/K|\big| |\mathrm{Out}(K/H)|$.

\end{cor}
\begin{proof} It follows straightforward from Lemma \ref{ty1} and Lemma \ref{ty6}$-$\ref{ty11}.

\end{proof}

\subsection{To prove $G\cong S$}

The necessity of the Main Theorem is obvious.  Now we start to prove   the sufficiency of The Main Theorem. In the following discussion, $K$ and $H$ always means subgroups of $G$ in Corollary \ref{ty12}.
We shall  frequently consider the possibilities  of $K/H$ by $|K/H|$ dividing $|S|$ by using the list of simple groups in \cite{ATLAS}. But  \cite{ATLAS}  lists only all the non-abelian simple groups of order less than $10^{25}$, except  that $L_2(q), L_3(q), U_3(q)$, $ L_4(q), U_4(q), S_4(q)$ and $ G_2(q)$, which are stopped at orders $10^6(q\leq 125), 10^{12}(q\leq31), 10^{12}(q\leq 32)$, $ 10^{16}(q\leq 11), 10^{16}(q\leq 11), 10^{16}(q\leq 41), 10^{20}(q\leq 25)$, respectively. Hence, there exist simple section $K/H$ not appeared in the list of \cite{ATLAS} when $|S|\geq 10^{6}$. Noticing the fact that orders of sporadic simple groups except $B$ and $M$ are all of order$<10^{25}$, $|B|\big||M|$ and  $10^{53}<|M|<10^{54}$. Hence, it is necessary to determine two kinds of  simple groups $T$ not contained in the list of \cite{ATLAS}, one is $T$ such that $|T|\big||M|$ and  $|T|<10^{54}$; another one is contained in the series of $L_2(q)(q>125), L_3(q)(q>31), U_3(q)(q>32)$, $ L_4(q)(q>11), U_4(q)(q>11), S_4(q)(q>41)$ and $ G_2(q)(q>25)$, such that $|T|$ dividing the order of some sporadic simple group except $B$ and $M$.  We have the following lemma.
\begin{lemma}\label{lemT}Let $S$ be a sporadic group, and $T$ a non-abelian simple group, then
	
(I) If $S=M$, $|T|\big||S|$, and $T$ is not contained in the list of \cite{ATLAS}, then $T$ is  one of $A_n(26\leq n\leq 32)$, $L_2(2^{10})$ and $L_2(13^2)$.

(II) If $S\neq$ $M$, $|T|\big||S|$,  then $T$ is contained in the list of \cite{ATLAS}. \end{lemma}
\begin{proof}
For a simple group $T$  satisfying the hypothesis of (I), since $10^{53}<|M|<10^{54}$, we have that $|T|\leq 10^{54}$, hence we have to check if there exists  $T$  of order $\leq 10^{25}$ among the series of  $L_2(q)(q>125), L_3(q)(q>31), U_3(q)(q>32)$, $ L_4(q)(q>11), U_4(q)(q>11), S_4(q)(q>41)$ and $ G_2(q)(q>25)$, and to check if there exists  $T$  of order greater than  $10^{25}$ and smaller than $10^{54}$ among all non-abelian simple groups except sporadic simple  groups. In order to make the proof easy to read, we check all series of non-abelian simple groups except sporadic simple groups. By  Classification Theorem of Simple Groups, we have the following steps.

(1) Assume that $T\cong A_n$, an alternating group. Note that $|A_{25}|<10^{25}<|A_{26}|$, we have $n\geq 26$. On the other hand, in view of the fact that $11^2$$\parallel$$ |M|$, we come to that $26\leq n\leq 32$. By calculating the order of  $A_n(26\leq n\leq 32)$, we get that $T\cong A_n(26\leq n\leq 32$).

(2) Assume that $T\cong A_n(q)$. Then  $|T|=q^{\frac{n(n+1)}{2}}\prod\limits_{i=1}^n(q^{i+1}-1)/(n+1,q-1)$. It follows from $|T|\big||M|$  and the largest exponent of prime power of $|M|$ is $46$ that $n\leq 9$. Therefore, $T$ can only be one of the Lie type simple groups: $A_1(q)(q>125)$, $A_2(q)(q>31)$, $A_3(q)(q>11)$, $A_4(q)(q>11)$, $A_5(q)(q>5)$, $A_6(q)(q>3)$, $A_7(q)(q>2)$, $A_8(q)(q>2)$, $A_9(q)(q<4)$.

If $T=A_1(q)(q>125)$, since   $|T|$ divides $|M|$, we deduce that one of the cases holds: $q=2^t(7\leq t\leq 46)$,  $q=3^t(5\leq t\leq 20)$, $q=5^t(4\leq t\leq 9)$, $q=7^t(3\leq t\leq 6)$, $q=13^t(2\leq t\leq 3)$. Now calculating the order of these groups by using Maple, we come to that $q$ is either $2^{10}$ or $13^2$.

If $T=A_2(q)(q>31)$, then $q^3$ divides one of  $2^{45}$, $3^{18}$, $5^9$ and $7^6$. Hence, $q$ is one of: $2^t(5\leq t\leq 15)$, $q=3^t(4\leq t\leq 6)$, $q=5^3$ and $q=7^2$. Now calculating the order of these groups, for example, by using Maple, we have $|T|\nmid |M|$, a contradiction.

By the same approach above, we can get contradictions if $T=A_n(q)$, where $3\leq n\leq 9$.

(3) To show $T\not\cong B_n(q)(n\geq 2)$ or $C_n(q)(n\geq 3)$. Otherwise, $|T|=|B_n(q)|=|C_n(q)|=q^{n^2}\prod\limits_{i=1}^n(q^{2i}-1)/(2,q-1)$. It follows from $|T|\big||M|$ that $n\leq 6$. Hence, $T$ may be one of $B_2(q)(q>41)$, $B_3(q)(q>13)$, $B_4(q)(q>5)$, $B_5(q)(q>2)$ and $B_6(q)(q>2)$.
Since the approach is the same, here we discuss the case that $T\cong B_2(q)(q>41)$ as an example, and the other cases can be studied by the same way. In view of the facts that $|T|\big||M|$, $q>41$ and the exponent  of $q$ is  $4$, by comparing the powers in $|M|$, we conclude that $q^{4}$ divides $2^{44}$ or $3^{20}$. Hence, $q=2^t,\ 6\leq t\leq 11$ or $q=3^t,\ 4\leq t\leq 5$. Now calculating the order of  each group, we get that $|M|\not\equiv0\ \ (\rm{mod}\ |T|)$, a contradiction.

(4) To show $T\not\cong D_n(q)(n\geq 4)$. Otherwise
 $|T|=q^{n(n-1)}(q^n-1)\prod\limits_{i=1}^{n-1}(q^{2i}-1)/(4,q^n-1)$. It follows from $|T|\big||M|$ that $n\leq 7$. Therefore, $T$ may be one of $D_4(q)(q>19)$, $D_5(q)(q>5)$, $D_6(q)(q>3)$, $D_7(q)(q>2)$. Here, as an example, assume that $T\cong D_4(q)(q>19)$, and the other cases can be investigated  similarly. In view of the facts that $|T|\big||M|$ and $n\geq 4$, we deduce that  $q^{12}\big||T|$, thus, $q$ is a power of $2$ or $3$.  But $q>19$, hence, $q=2^r$ or $3^s$, where $r\geq 5$ and $s\geq 3$, which implies $2^{60}$ or $3^{36}$ divides $|T|$, a contradiction.

 (5) To show $T\not\cong G_2(q)(q>25)$, $F_4(q)(q>3)$, $E_6(q)(q>2)$.  Since $q^{24}$ and $q^{36}$ divides the orders of $F_4(q)(q>3)$ and $E_6(q)(q>2)$, respectively, so it is easy to show $T\not\cong F_4(q)(q>3)$, $E_6(q)(q>2)$. Assume $T\cong G_2(q)(q>25)$, then $|T|=q^6(q^6-1)(q^2-1)$. Comparing the power of a prime in $|M|$ and noticing $q>25$, we get that $q$ equals  one of $2^t,$ $5\leq t\leq 7$, $3^3$. Now calculating the orders of  these groups, we get that $|M|\not\equiv0\ \ (\rm{mod}\ |T|)$, a contradiction.

 (6) To Show $T\not\cong E_7(q)$ and $E_8(q)$. Since $q^{63}\big||E_7(q)|$ and $q^{120}\big||E_8(q)|$, we have $|T|$ can not divide $|M|$.

 (7) To show $T\not\cong $$^2$$A_n(q)(n\geq 2)$. Assume $T\cong $$^2$$A_n(q)(n\geq 2)$. Then
 $|T|=q^{\frac{n(n+1)}{2}}\prod\limits_{i=1}^{n}(q^{i+1}-(-1)^{i+1})/(n+1,q+1)$. It follows from $|T|\big||M|$ that $n\leq 9$. Therefore, $T$ may be one of $^2$$A_2(q)(q>32)$, $^2$$A_3(q)(q>11)$, $^2$$A_4(q)(q>11)$, $^2$$A_5(q)(q>5)$, $^2$$A_6(q)(q>3)$, $^2$$A_i(q)(i=7,8, q>2)$ and $^2$$A_9(q)(q<4)$. Here, we take $T\cong$$^2$$A_2(q)(q>32)$ as an example, and the other cases may be studied similarly. In view of the facts that $|T|\big||M|$ and $q>32$, we deduce that $q^{3}$ may be a power of $2$, $3$, $5$ or $7$ and divides $2^{45}$, $3^{18}$,  $5^9$ or $7^6$. Hence, $q$ equals one of $2^{t},6\leq t\leq 15$, $3^{t},4\leq t\leq 6$, $5^3$ and $7^2$. By calculation, we get that $|M|\not\equiv0\ \ (\rm{mod}\ |T|)$, a contradiction.

 (8) To show $T\not\cong$ $^2$$ B_2(q),q=2^{2m+1}(m>7)$. Otherwise, $T\cong$ $^2$$ B_2(q),q=2^{2m+1}(m>7)$. Then  $|T|=q^2(q^2+1)(q-1)$. Comparing the power of $2$ in $|M|$, we get that $q$ equals  one of $2^t,\ t=2m+1,$ $8\leq m\leq 22$. By  calculation, we get that $|M|\not\equiv0\ \ (\rm{mod}\ |T|)$, a contradiction.

  (9) To show $T\not\cong$$^2$$D_n(q)(n\geq4)$. Otherwise, $T\cong$$^2$$D_n(q)(n\geq4)$. Then $|T|=q^{n(n-1)}(q^n+1)\prod\limits_{i=1}^{n-1}(q^{2i}-1)/(4,q^n+1)$,  hence, $n\leq 7$. Therefore, $T$ may be one of $^2$$D_4(q)(q>17)$, $^2$$D_5(q)(q>7)$, $^2$$D_6(q)(q>3)$ and  $^2$$D_7(q)(q>2)$. We take $T\cong$$^2$$D_4(q)(q>17)$  as an example, and the other cases may be dealt with similarly. If $T\cong$$^2$$D_4(q)(q>17)$, then $q^{12}\big||T|$, so that $q^{12}$ divides $2^{46}$ or $3^{20}$ by observing the powers of primes dividing $|M|$. Hence, $q=2^r$ or $3^s$ since $q>17$, where $r\geq 5$ and $s\geq 3$. Therefore $2^{60}$ or $3^{36}$ dividing $|T|$, a contradiction.

 (10) To show $T\not\cong$ $^3$$ D_4(q)(q>7)$. Otherwise, $T\cong$ $^3$$ D_4(q)(q>7)$. Then  $|T|=q^{12}(q^8+q^4+1)(q^6-1)(q^2-1)$, so $q^{12}\big||M|$. Since $q>7$, comparing the prime power of $|M|$, it follows that $q=2^3$, and $T=^3$$D_4(8)$, we get that $|M|\not\equiv0\ \ (\rm{mod}\ |T|)$, a contradiction.

 (11) To show $T\not\cong$ $^2$$ G_2(q)(q=3^{2m+1}, m>3)$. Otherwise, $T\cong$ $^2$$ G_2(q)(q=3^{2m+1}, m>3)$. Then $|T|=q^{3}(q^3+1)(q-1)$ and $q^{3}\big|3^{20}$, hence, $q\big|3^6$, contradicting $q\geq 3^9$.

 (12) To show that $T\not\cong ^2$$ F_4(q)(q=2^{2m+1})$, $^2$$ E_6(q)$. Assume $T\cong$ $^2$$ F_4(q)(q=2^{2m+1})$, $^2$$ E_6(q)$. Then $|T|=q^{12}(q^6+1)(q^4-1)(q^3+1)(q-1)$, and  $q^{12}$ divides $2^{46}$, thus, $q=2^3$, but $|M|\not\equiv0\ \ (\rm{mod}\ |^2F_4(8)|)$, a contradiction. Using the similar arguments as above, we can easily get a contradiction if $T\cong$$^2$$ E_6(q)$.

  By (1)-(12), we come to that (I) follows.

  Now we start to prove (II). Let $S$ be a sporadic group, then $|S|<10^{25}$ if  $S\neq B$ and $M$;  and $|S|\big||M|$ if  $S\neq Ly$ and $J_4$. Hence if $S\not\cong Ly$ and $J_4$, then $T$ is  among simple groups discribed in  (I), i.e., $L_2(2^{10})$, $L_2(13^2)$ and $A_n\ (26\leq n\leq 32)$. By  checking, it follows that no one of   $|L_2(2^{10})|$, $|L_2(13^2)|$ and $|A_n|\ (26\leq n\leq 32)$ can  divide the order of any sporadic group except  $M$, which implies that (II) follows for $S\neq M$, $Ly$ and $J_4$. If $S\cong Ly$ or $J_4$,   $T$ is one of $L_2(q)(q>125), L_3(q)(q>31), U_3(q)(q>32)$, $ L_4(q)(q>11), U_4(q)(q>11), S_4(q)(q>41)$ and $ G_2(q)(q>25)$, then, by using Maple to calculate the orders of groups with orders $<10^{25}$ among these series, we can show that  $|T|$ does not divide $|S|$. So (II) and then Lemma \ref{lemT} follows.\\
\end{proof}

\par Now we begin to prove $G\cong S$, we divide the whole proof into six  lemmas.

 \begin{lemma} \label{le1} Let $G$ be a group and    $S$   one of groups: $M_{11}$, $M_{12}$, $M_{22}$, $M_{23}$,  $M_{24}$ and $J_2$.  If $|G|=|S|$ and the prime graph of $G$ is disconnected, then $G\cong S$.\end{lemma}
 \begin{proof}
 By assumptions, $|G|$ is one of $|M_{11}|,\ |M_{12}|,\ |M_{22}|,\
 |M_{23}|$, $
 |M_{24}|$ and $|J_2|$,  also, $\Gamma(G)$ is disconnected. Applying  Corollary \ref{ty12}, $G$ has a unique simple section $K/H$ such that  $K$ and $H $ satisfying (*). By Lemma 2.15, $K/H$ can only be simple groups in the list of \cite{ATLAS}. By comparing the order of $G$ with the orders of the simple groups in \cite{ATLAS}, we can obtain  all possibilities of  $K/H$. We discuss step by step  following what $S$ is.

 \textbf{Step 1} If $S= M_{11}$, then $G\cong M_{11}$.

 By $|G|=|M_{11}|=2^4\cdot 3^2\cdot5\cdot11$, checking simple groups in the list of \cite{ATLAS}, we get that $K/H$ is isomorphic to one of groups in Table \ref{tb:M11}.

\begin{table}[h]
\setlength{\abovecaptionskip}{10pt}
\setlength{\belowcaptionskip}{0pt}
\centering
\caption{The non-abelian simple groups $K/H$ with orders dividing $|M_{11}|$~}\label{tb:M11}

\begin{tabular}
{p{5.5cm}<{\centering}p{5.5cm}<{\centering}p{4.5cm}<{\centering}p{3.5cm}<{\centering}}
  \hline
$K/H$      & $|K/H|$
  & $|\mathrm{Out}(K/H)|$     \\
  \hline
  $A_5$       & $2^{2}\cdot3\cdot5$ & $2$ \\
  $A_6$       & $2^{3}\cdot3^2\cdot5$ & $4$ \\
  $L_2(11)$   & $2^{2}\cdot3\cdot5\cdot11$ & $2$\\
  $M_{11}$    & $2^4\cdot3^2\cdot5\cdot11$ & $1$   \\
  \hline
\end{tabular}

\end{table}

If $K/H\cong A_5$ or $ A_6$, then $11\in \pi(H)$, and $|H_{11}|=11$. Since $H$ is nilpotent, we know $H_{11}\unlhd G$. Note that $|G/C_G(H_{11})|\big||\mathrm{Aut}(H_{11})|$, one has that $\{2, 3,  11\}\subseteq \pi(C_G(H_{11}))$.  Since $\Gamma(G)$ is disconnected, we have  $5\notin\pi(C_G(H_{11}))$. Consider the action on $C_G(H_{11})$ by an element $g$ of order $5.$ Applying Lemma 2.4,  $C_G(H_{11})$ has a $<$$g$$>-$invariant $Sylow\ 3-$subgroup, but  $5\nmid|GL(2,3)|$, thus, the action is trivial, and then $15\in\pi_e(G)$, which means  $\Gamma(G)$ is connected, a contradiction.

If $K/H\cong L_2(11)$, then $|H_3|=3$, $H_3\unlhd G$. We get $t(G)=1$ by $|G/C_G(H_3)|\big|2$, a contradiction.

Now we have $K/H\cong M_{11},$ so $G\cong S\cong M_{11 }$ by $|G|=|M_{11}|,$ as desired.

\textbf{Step 2} If $S= M_{12}$, then $G\cong M_{12}$.

Since $|G|=|M_{12}|=2^6\cdot 3^3\cdot5\cdot11$,  we get by \cite{ATLAS}  that $K/H$ is isomorphic to one of the  groups in Table \ref{tb:M12}.

\begin{table}[h]
\setlength{\abovecaptionskip}{10pt}
\setlength{\belowcaptionskip}{0pt}
\centering
\caption{The non-abelian simple groups $K/H$ with orders dividing $|M_{12}|$~}\label{tb:M12}
\begin{tabular}
{p{5.5cm}<{\centering}p{5.5cm}<{\centering}p{4.5cm}<{\centering}p{3.5cm}<{\centering}}
  \hline
$K/H$      & $|K/H|$
  & $|\mathrm{Out}(K/H)|$      \\
  \hline
  $A_5$       & $2^{2}\cdot3\cdot5$ & $2$ \\
  $A_6$       & $2^{3}\cdot3^2\cdot5$ & $4$ \\
  $L_2(11)$   & $2^{2}\cdot3\cdot5\cdot11$ & $2$\\
  $M_{11}$    & $2^4\cdot3^2\cdot5\cdot11$ & $1$   \\
  $M_{12}$    & $2^6\cdot3^3\cdot5\cdot11$ & $2$   \\
  \hline
\end{tabular}

\end{table}

 If $K/H$ is isomorphic to one of the groups in Table \ref{tb:M12} except  $M_{12}$, then $|(K/H)_3|\big|3^2$ and $3\nmid|\rm{Out}(K/H)|$, hence, $3\in\pi(H)$. Clearly,   $|H_3|=3$ or $3^2$,  also $H_3\unlhd G$. Note that $|G/C_G(H_3)|\big||\mathrm{Aut}(H_3)|,$ and $|\mathrm{Aut}(H_3)|\big|2^4\cdot3,$ we have  $\pi(G)=\pi(C_G(H_3)),$ thus, $\Gamma(G)$ is connected, a contradiction.

Now we have $K/H\cong M_{12},$ so that $G\cong S \cong M_{12}$, as desired.

\textbf{Step 3} If $S= M_{22}$, then $G\cong M_{22}$.

Since $|G|=|M_{22}|=2^7\cdot 3^2\cdot5\cdot7\cdot11$. by the same reason, we get that $K/H$ is isomorphic to one of the  groups in Table \ref{tb:M22}.

\begin{table}[h]
\setlength{\abovecaptionskip}{10pt}
\setlength{\belowcaptionskip}{0pt}
\centering
\caption{The non-abelian simple groups $K/H$ with orders dividing $|M_{22}|$~}\label{tb:M22}
\begin{tabular}
{p{5.5cm}<{\centering}p{5.5cm}<{\centering}p{4.5cm}<{\centering}p{3.5cm}<{\centering}}
  \hline
$K/H$      & $|K/H|$
  & $|\mathrm{Out}(K/H)|$      \\
  \hline
  $A_5$       & $2^{2}\cdot3\cdot5$ & $2$ \\
  $L_3(2)$       & $2^{3}\cdot3\cdot7$ & $2$ \\
  $A_6$       & $2^{3}\cdot3^2\cdot5$ & $4$ \\
  $L_2(8)$       & $2^{3}\cdot3^2\cdot7$ & $3$ \\
  $L_2(11)$   & $2^{2}\cdot3\cdot5\cdot11$ & $2$\\
  $A_7$   & $2^{3}\cdot3^2\cdot5\cdot7$ & $2$\\
  $M_{11}$    & $2^4\cdot3^2\cdot5\cdot11$ & $1$   \\
  $A_8$   & $2^{6}\cdot3^2\cdot5\cdot7$ & $2$\\
  $L_3(4)$   & $2^{6}\cdot3^2\cdot5\cdot7$ & $2^2\cdot3$\\
  $M_{22}$    & $2^7\cdot3^2\cdot5\cdot7\cdot11$ & $2$   \\
  \hline
\end{tabular}

\end{table}

From  Table \ref{tb:M22}, we conclude that $11\in\pi(H)$ while $K/H$ is isomorphic to one of the groups: $A_5$, $L_3(2)$, $A_6$, $L_2(8)$, $A_7$, $A_8$ and $L_3(4)$.  By using the same arguments as the second paragraph in Step 1, we also get  contradictions.

If $K/H\cong$$L_2(11)$ or $M_{11}$, then we see  that $7\in\pi(H).$  Surely,  $|H_7|=7$ and  $H_7\unlhd G$. So $\pi(G)=\pi(C_G(H_7))$ for $|G/C_G(H_7)|\big||\mathrm{Aut}(H_7)|$. Consequently,  $\Gamma(G)$ is connected, a contradiction.

Now we have $K/H\cong M_{22},$ then $G\cong S\cong M_{22}$, as desired.

\textbf{Step 4} If $S= M_{23}$, then $G\cong M_{23}$.

In this case, $|G|=|M_{23}|=2^7\cdot 3^2\cdot5\cdot7\cdot11\cdot23$. By \cite{ATLAS}, we get that $K/H$ is isomorphic to one of the  groups in Table \ref{tb:M23}.

\begin{table}[H]
\setlength{\abovecaptionskip}{10pt}
\setlength{\belowcaptionskip}{0pt}
\centering
\caption{The non-abelian simple groups $K/H$ with orders dividing $|M_{23}|$~}\label{tb:M23}
\begin{tabular}
{p{5.5cm}<{\centering}p{5.5cm}<{\centering}p{4.5cm}<{\centering}p{3.5cm}<{\centering}}
  \hline
 $K/H$      & $|K/H|$
  & $|\mathrm{Out}(K/H)|$      \\
  \hline
  $A_5$       & $2^{2}\cdot3\cdot5$ & $2$ \\
  $L_3(2)$       & $2^{3}\cdot3\cdot7$ & $2$ \\
  $A_6$       & $2^{3}\cdot3^2\cdot5$ & $4$ \\
  $L_2(8)$       & $2^{3}\cdot3^2\cdot7$ & $3$ \\
  $L_2(11)$   & $2^{2}\cdot3\cdot5\cdot11$ & $2$\\
  $A_7$   & $2^{3}\cdot3^2\cdot5\cdot7$ & $2$\\
  $L_2(23)$   & $2^{3}\cdot3\cdot11\cdot23$ & $2$\\
  $M_{11}$    & $2^4\cdot3^2\cdot5\cdot11$ & $1$   \\
  $A_8$   & $2^{6}\cdot3^2\cdot5\cdot7$ & $2$\\
  $L_3(4)$   & $2^{6}\cdot3^2\cdot5\cdot7$ & $2^2\cdot3$\\
  $M_{22}$    & $2^7\cdot3^2\cdot5\cdot7\cdot11$ & $2$   \\
  $M_{23}$    & $2^7\cdot3^2\cdot5\cdot7\cdot11\cdot23$ & $1$   \\
  \hline
\end{tabular}

\end{table}

 If $K/H$ is isomorphic to one of the groups in Table \ref{tb:M23} except   $L_2(23)$ and $M_{23}$, then $23\in\pi(H)$. Clearly,    $|H_{23}|=23$ and  $H_{23}\unlhd G.$  Since $|G/C_G(H_{23})|\big||\mathrm{Aut}(H_{23})|$, it follows that $\{2,\ 3,\ 5,\ 7,\ 23\}\subseteq \pi(C_G(H_{23})).$ Taking into account the fact that $\Gamma(G)$ is not connected, we have  $11\notin\pi(C_G(H_{23}))$. Now consider  the action of an element $g$ of order $11$ on $C_G(H_{23})$ and  applying Lemma \ref{ty4},  we see that $C_G(H_{23})$ has a $\left\langle g\right\rangle-$invariant $Sylow\ 7-$subgroup,  which implies that  $77\in\pi_e(G)$, therefore,  $\Gamma(G)$ is connected, a contradiction.

 Suppose that $K/H\cong L_2(23)$, we deduce that $7\in\pi(H)$. Surely,  $|H_7|=7$ and $H_7\unlhd G$. Furthermore, we conclude that  $\pi(G)=\pi(C_G(H_7))$ for $|G/C_G(H_7)|\big||\mathrm{Aut}(H_7)|.$ Consequently,  $\Gamma(G)$ is connected, a contradiction.

Now we have  $K/H\cong M_{23},$ then $G\cong S\cong M_{23}$, as desired.

\textbf{Step 5} If $S= M_{24}$ or $J_2$, then $G\cong S$.

In these cases, either  $|G|=|M_{24}|=2^{10}\cdot 3^3\cdot5\cdot7\cdot11\cdot23$ or $|G|=|J_{2}|=2^7\cdot 3^3\cdot5^2\cdot7$. By Lemma \ref{lemT}, we check the list of simple groups in \cite{ATLAS} and come to that $K/H$ is a simple group in Table \ref{tb:M24} and \ref{tb:J2}, respectively.

\begin{table}[h]
\setlength{\abovecaptionskip}{10pt}
\setlength{\belowcaptionskip}{0pt}
\centering
\caption{The non-abelian simple groups $K/H$ with orders dividing $|M_{24}|$~}\label{tb:M24}
\begin{tabular}
{p{5.5cm}<{\centering}p{5.5cm}<{\centering}p{4.5cm}<{\centering}p{3.5cm}<{\centering}}
  \hline
  $K/H$      & $|K/H|$
  & $|\mathrm{Out}(K/H)|$   \\
  \hline
  $A_5$       & $2^{2}\cdot3\cdot5$ & $2$ \\
  $L_3(2)$       & $2^{3}\cdot3\cdot7$ & $2$ \\
  $A_6$       & $2^{3}\cdot3^2\cdot5$ & $4$ \\
  $L_2(8)$       & $2^{3}\cdot3^2\cdot7$ & $3$ \\
  $L_2(11)$   & $2^{2}\cdot3\cdot5\cdot11$ & $2$\\
  $A_7$   & $2^{3}\cdot3^2\cdot5\cdot7$ & $2$\\
  $U_3(3)$       & $2^{5}\cdot3^3\cdot7$ & $2$ \\
  $L_2(23)$   & $2^{3}\cdot3\cdot11\cdot23$ & $2$\\
  $M_{11}$    & $2^4\cdot3^2\cdot5\cdot11$ & $1$   \\
  $A_8$   & $2^{6}\cdot3^2\cdot5\cdot7$ & $2$\\
  $L_3(4)$   & $2^{6}\cdot3^2\cdot5\cdot7$ & $2^2\cdot3$\\
  $M_{22}$    & $2^7\cdot3^2\cdot5\cdot7\cdot11$ & $2$   \\
  $M_{23}$    & $2^7\cdot3^2\cdot5\cdot7\cdot11\cdot23$ & $1$   \\
  $M_{24}$    & $2^{10}\cdot3^3\cdot5\cdot7\cdot11\cdot23$ & $1$   \\
  \hline
\end{tabular}

\end{table}

\begin{table}[h]
\setlength{\abovecaptionskip}{10pt}
\setlength{\belowcaptionskip}{0pt}
\centering
\caption{The non-abelian simple groups $K/H$ with orders dividing $|J_{2}|$~}\label{tb:J2}
\begin{tabular}
{p{5.5cm}<{\centering}p{5.5cm}<{\centering}p{4.5cm}<{\centering}p{3.5cm}<{\centering}}
  \hline
 $K/H$       & $|K/H|$
  & $|\mathrm{Out}(K/H)|$     \\
  \hline
  $A_5$       & $2^{2}\cdot3\cdot5$ & $2$ \\
  $L_2(7)$       & $2^{3}\cdot3\cdot7$ & $2$ \\
  $A_6$       & $2^{3}\cdot3^2\cdot5$ & $4$ \\
  $L_2(8)$   & $2^{3}\cdot3^2\cdot7$ & $3$\\
  $A_7$   & $2^{3}\cdot3^2\cdot5\cdot7$ & $2$\\
  $U_{3}(3)$    & $2^5\cdot3^3\cdot7$ & $2$   \\
  $A_{8}$    & $2^6\cdot3^2\cdot5\cdot7$ & $2$ \\
  $L_{3}(4)$    & $2^6\cdot3^2\cdot5\cdot7$ & $2^2\cdot3$\\
  $J_{2}$    & $2^7\cdot3^3\cdot5^2\cdot7$ & $2$   \\
  \hline
\end{tabular}

\end{table}

 Suppose that $K/H$ is isomorphic to one of the groups in Table \ref{tb:M24} and Table \ref{tb:J2} except $L_2(8)$, $L_3(4)$, $U_3(3)$ and $S$, then $|(K/H)_3|\big|3^2$ and $3\nmid |\rm{Out}(K/H)|$, hence, $3\in\pi(H)$. Moreover,  $|H_3|=3$ or $3^2.$  By using the same arguments as the second paragraph in Step 2, we can get  contradictions.

Suppose that $K/H\cong L_2(8)$ or $U_{3}(3)$, then $5\in\pi(H)$ and $|H_5|=5$ or $5^2$. Clearly,   $H_5\unlhd G.$ Thus, we get that  $t(G)=1$ for $|G/C_G(H_5)|\big||\mathrm{Aut}(H_5)|$ and $|\rm{GL}(2,5)|=2^5\cdot3\cdot5$. Consequently, $\Gamma(G)$ is connected, a contradiction.

Suppose that $K/H\cong L_3(4)$,  then $11\in\pi(H)$.  Note that $5\nmid|\rm{GL(3,3)}|$, using the same argument as the second paragraph in Step 1, we also get  a contradiction.

Now we have  $K/H\cong S$, then $G\cong S$ by $|G|=|K/H|$, as desired.
\end{proof}

\begin{lemma} \label{le2} Let $G$ be a group and    $S$   one of  groups: $J_1$, $J_3$ and $J_4$.  If $|G|=|S|$ and the prime graph of $G$ is disconnected, then $G\cong S$.\end{lemma}
\begin{proof} Similar to Lemma \ref{le1}, we need consider the possibilities of the unique simple section $K/H$ for each $S$, and  have the following steps.

\textbf{Step 1} If $S= J_1$, then $G\cong J_1$.

If $S= J_1$, then $|G|=|J_{1}|=2^3\cdot3\cdot5\cdot7\cdot11\cdot19$.  Applying Lemma \ref{lemT}, we get by \cite{ATLAS} that $K/H$ is isomorphic to one of the  groups in Table \ref{tb:J1}.

\begin{table}[h]
\setlength{\abovecaptionskip}{10pt}
\setlength{\belowcaptionskip}{0pt}
\centering
\caption{The non-abelian simple groups $K/H$ with orders dividing $|J_{1}|$~}\label{tb:J1}

\begin{tabular}
{p{5.5cm}<{\centering}p{5.5cm}<{\centering}p{4.5cm}<{\centering}p{3.5cm}<{\centering}}
  \hline
  $K/H$      & $|K/H|$
  & $|\mathrm{Out}(K/H)|$     \\
  \hline
  $A_5$       & $2^{2}\cdot3\cdot5$ & $2$ \\
  $L_3(2)$       & $2^{3}\cdot3\cdot7$ & $2$ \\
  $L_2(11)$   & $2^{2}\cdot3\cdot5\cdot11$ & $2$\\
  $J_{1}$    & $2^3\cdot3\cdot5\cdot7\cdot11\cdot19$ & $1$   \\
  \hline
\end{tabular}

\end{table}

If $K/H$ is isomorphic to one of $ A_5$, $L_3(2)$ and $L_2(11)$, then $19\in \pi(H)$ and $|H_{19}|=19$. Since  $H_{19}\unlhd G$, we have $|G/C_G(H_{19})|\big|2\cdot3^2$, thus, $\{2,5,7,11,19\}\subseteq \pi(C_G(H_{19}))$. The fact  that $\Gamma(G)$ is disconnected yields $3\not\in\pi(C_G(H_{19}))$. Let  a $3-$element $g$ act on $C_G(H_{19})$, then, by Lemma \ref{ty4}, $C_G(H_{19})$ has a $\left\langle g\right\rangle-$invariant  $\rm{Sylow}\ 5-$subgroup. Since $5\parallel |G|$, which implies that $15\in\pi_e(G)$, therefore, $\Gamma(G)$ is connected, a contradiction.

Now we have $K/H\cong J_{1},$ so $G\cong S\cong J_{1 }$ by $|G|=|J_{1}|,$ as desired.

\textbf{Step 2} If $S= J_3$, then $G\cong J_3$.

In this case, $|G|=|J_{3}|=2^7\cdot 3^5\cdot5\cdot17\cdot19$. Applying Lemma \ref{lemT}, we get by \cite{ATLAS} that  $K/H$ is isomorphic to one of the  groups in Table \ref{tb:J3}.

\begin{table}[h]
\setlength{\abovecaptionskip}{10pt}
\setlength{\belowcaptionskip}{0pt}
\centering
\caption{The non-abelian simple groups with orders dividing $|J_{3}|$~}\label{tb:J3}
\begin{tabular}
{p{5.5cm}<{\centering}p{5.5cm}<{\centering}p{4.5cm}<{\centering}p{3.5cm}<{\centering}}
  \hline
  $K/H$       & $|K/H|$
  & $|\mathrm{Out}(K/H)|$     \\
  \hline
  $A_5$       & $2^{2}\cdot3\cdot5$ & $2$ \\
  $A_6$       & $2^{3}\cdot3^2\cdot5$ & $4$ \\
  $L_2(17)$   & $2^{4}\cdot3^2\cdot17$ & $2$\\
  $L_2(19)$   & $2^{2}\cdot3^2\cdot5\cdot19$ & $2$\\
  $L_{2}(16)$    & $2^4\cdot3\cdot5\cdot17$ & $4$   \\
  $U_4(2)$   & $2^{6}\cdot3^4\cdot5$  & $2$\\
  $J_{3}$    & $2^7\cdot3^5\cdot5\cdot17\cdot19$ & $2$   \\
  \hline
\end{tabular}

\end{table}

 While $K/H$ is isomorphic to one of the groups in Table \ref{tb:J3} except $L_2(19)$ and $J_3$, we conclude that $19\in\pi(H)$. Then we deduce that $t(G)=1$ since $|G/C_G(H_{19})|\big|2\cdot3^2$, a contradiction.

 Suppose that $K/H\cong L_2(19)$, then $17\in\pi(H)$, also we get a contradiction by $t(G)=1$ as $|G/C_G(H_{17})|\big|2^4$.

 Therefore, $K/H\cong J_3$, since $|G|=|K/H|=|J_3|$, then  $G\cong S\cong J_3.$

\textbf{Step 3} If $S= J_4$, then $G\cong J_4$.

 If $S= J_4$, then $|G|=|J_{4}|=2^{21}\cdot 3^3\cdot5\cdot7\cdot11^3\cdot23\cdot29\cdot31\cdot37\cdot43$. Applying Lemma \ref{lemT}, we get by \cite{ATLAS} that $K/H$ is isomorphic to one of the  groups in Table \ref{tb:J4}.

\begin{table}[H]
\setlength{\abovecaptionskip}{10pt}
\setlength{\belowcaptionskip}{0pt}
\centering
\caption{The non-abelian simple groups $K/H$ of orders dividing $|J_{4}|$~}\label{tb:J4}
\begin{tabular}{ccc|ccc}
\hline
  $K/H$       & $|K/H|$
  & $|\mathrm{Out}(K/H)|$   & $K/H$       & $|K/H|$
  & $|\mathrm{Out}(K/H)|$ \\
  \hline
   $A_5$       & $2^{2}\cdot3\cdot5$ & $2$ & $A_8$  &  $2^{6}\cdot3^2\cdot5\cdot7$ & $2$ \\
  $L_3(2)$       & $2^{3}\cdot3\cdot7$ & $2$ &$L_3(4)$    & $2^{6}\cdot3^2\cdot5\cdot7$ & $2^2\cdot3$ \\
  $A_6$       & $2^{3}\cdot3^2\cdot5$ & $4$ & $L_2(32)$  &$2^{5}\cdot3\cdot11\cdot31$ & $5$\\
  $L_2(8)$       & $2^{3}\cdot3^2\cdot7$ & $3$  & $L_2(43)$   & $2^{2}\cdot3\cdot7\cdot11\cdot43$& $2$\\
  $L_2(11)$   & $2^{2}\cdot3\cdot5\cdot11$ & $2$ & $M_{12}$  &$2^{6}\cdot3^3\cdot5\cdot11$ & $2$\\
  $A_7$   & $2^{3}\cdot3^2\cdot5\cdot7$ & $2$ & $M_{22}$   &$2^7\cdot3^2\cdot5\cdot7\cdot11$ & $2$\\
  $U_3(3)$   & $2^{5}\cdot3^3\cdot7$ & $2$ &$L_{5}(2)$   &$2^{10}\cdot3^2\cdot5\cdot\cdot7\cdot31$ &$2$\\
  $L_2(23)$   & $2^{3}\cdot3\cdot11\cdot23$ & $2$ &$M_{23}$   & $2^7\cdot3^2\cdot5\cdot7\cdot11\cdot23$& $1$\\
  $M_{11}$    & $2^4\cdot3^2\cdot5\cdot11$ & $1$ &$U_{3}(11)$   &$2^{5}\cdot3^2\cdot5\cdot11^3\cdot37$ &$2\cdot3$ \\
  $L_{2}(29)$    & $2^2\cdot3\cdot5\cdot7\cdot29$ & $2$   & $M_{24}$   &$2^{10}\cdot3^3\cdot5\cdot7\cdot11\cdot23$ & $1$\\
  $L_{2}(31)$    & $2^5\cdot3\cdot5\cdot31$ & $2$ & $J_4$   &$2^{21}\cdot3^3\cdot5\cdot7\cdot11^3\cdot23\cdot29\cdot31\cdot37\cdot43$ & $1$ \\
 \hline
\end{tabular}
\end{table}

 If $K/H$ is isomorphic to one of the groups in Table \ref{tb:J4} except   $L_2(43)$ and $J_{4}$, then $43\in\pi(H)$ and $|H_{43}|=43$. Also $H_{43}\unlhd G.$  Note that $|G/C_G(H_{43})|\big||\mathrm{Aut}(H_{43})|$, we have $\pi(G)\smallsetminus\{7\}\subseteq \pi(C_G(H_{43})).$ Since $\Gamma(G)$ is not connected, we see that   $7\notin\pi(C_G(H_{43}))$. Considering the action on $C_G(H_{43})$ by an element $g$ of order $7$. Lemma \ref{ty4} indicates that $C_G(H_{43})$ has a $\left\langle g\right\rangle-$invariant $Sylow\ 23-$subgroup,  which implies   $23\cdot7\in\pi_e(G)$, therefore,  $\Gamma(G)$ is connected, a contradiction.

 If $K/H\cong L_2(43)$, then $5\in\pi(H)$ and $|H_5|=5$. We immediately get a contradiction by $t(G)=1$ since $|G/C_G(H_5)|\big|2^2$.

At last, $K/H\cong J_{4},$ then $G\cong S\cong J_{4}$, as desired.
\end{proof}
\begin{lemma} \label{le3}Let $G$ be a group and   $S$  one of $Co_{1}$, $Co_{2}$, $Co_3$, $Fi_{23}$.  If $|G|=|S|$ and the prime graph of $G$ is disconnected, then $G\cong S$.\end{lemma}
\begin{proof} We write the proof upon what $S$ is step by step.

\textbf{Step 1} If $S=Co_{1}$, then $G\cong Co_1$.

 By assumption, $|G|=|Co_{1}|=2^{21}\cdot 3^9\cdot5^4\cdot7^2\cdot11\cdot13\cdot23$. Applying Lemma \ref{lemT} and checking \cite{ATLAS}, we get that $K/H$ is isomorphic to one of the  groups in Table \ref{tb:Co1}.

\begin{table}[h]
\setlength{\abovecaptionskip}{10pt}
\setlength{\belowcaptionskip}{0pt}
\centering
\caption{The non-abelian simple groups $K/H$ of orders dividing $|Co_1|$~}\label{tb:Co1}

\begin{tabular}{ccc|ccc}
 \hline
  $K/H$       & $|K/H|$
  & $|\mathrm{Out}(K/H)|$&$K/H$       & $|K/H|$
  & $|\mathrm{Out}(K/H)|$     \\
  \hline
  $A_5$       & $2^{2}\cdot3\cdot5$ & $2$ &$G_2(3)$&$2^6\cdot3^6\cdot7\cdot13$&$2$\\
  $L_2(7)$ & $2^3\cdot3\cdot7$&$2$&$S_4(5)$&$2^6\cdot3^2\cdot5^4\cdot13$&$2$\\
  $A_6$       & $2^{3}\cdot3^2\cdot5$ & $4$ & $L_4(3)$&$2^7\cdot3^6\cdot5\cdot13$&$2^2$\\
  $L_2(8)$& $2^{3}\cdot3^2\cdot7$ & $3$  &$M_{23}$   & $2^7\cdot3^2\cdot5\cdot7\cdot11\cdot23$& $1$\\
  $L_2(11)$   & $2^{2}\cdot3\cdot5\cdot11$ & $2$ &$U_5(2)$&$2^{10}\cdot3^5\cdot5\cdot11$ &$2$\\
  $L_2{(13)}$&$2^2\cdot3\cdot7\cdot13$&$2$&$^2F_4(2)^{\prime}$&$2^{11}\cdot3^3\cdot5^2\cdot13$&$2$\\
  $A_7$   & $2^{3}\cdot3^2\cdot5\cdot7$ & $2$ &$A_{11}$&$2^7\cdot3^4\cdot5^2\cdot7\cdot11$&$2$\\
  $L_3(3)$&$2^4\cdot3^3\cdot13$&$2$&$L_3(9)$&$2^7\cdot3^6\cdot5\cdot7\cdot13$&$2^2$\\
  $U_3(3)$   & $2^{5}\cdot3^3\cdot7$ & $2$ &$HS$&$2^9\cdot3^2\cdot5^3\cdot7\cdot11$&$2$\\
  $L_2(23)$   & $2^{3}\cdot3\cdot11\cdot23$ & $2$ &$O_8^{+}(2)$&$2^{12}\cdot3^5\cdot5^2\cdot7$&$2\cdot3$\\
  $L_2(25)$&$2^3\cdot3\cdot5^2\cdot13$&$2^2$&$^3D_4{2}$&$2^{12}\cdot3^4\cdot7^2\cdot13$&$3$\\
  $M_{11}$    & $2^4\cdot3^2\cdot5\cdot11$ & $1$ &$A_{12}$&$2^9\cdot3^5\cdot5^2\cdot7\cdot11$&$2$\\
  $L_2(27)$&$2^2\cdot3^3\cdot7\cdot13$&$2\cdot3$&$M_{24}$&$2^{10}\cdot3^3\cdot5\cdot7\cdot11\cdot23$&$1$\\
  $A_8$  &  $2^{6}\cdot3^2\cdot5\cdot7$ & $2$ &$G_2(4)$&$2^{12}\cdot3^3\cdot5^2\cdot7\cdot13$&$2$\\ $M^cL$&$2^7\cdot3^6\cdot5^3\cdot7\cdot11$&$2$&$S_4(8)$&$2^{12}\cdot3^4\cdot5\cdot7^2\cdot13$&$2\cdot3$\\
  $L_3(4)$    & $2^{6}\cdot3^2\cdot5\cdot7$ & $2^2\cdot3$& $A_{13}$   &$2^{9}\cdot3^5\cdot5^2\cdot7\cdot11\cdot13$ & $2$ \\
  $U_{4}(2)$    & $2^2\cdot3^4\cdot5$ & $2$&$S_6(3)$&$2^9\cdot3^9\cdot5\cdot7\cdot13$&$2$\\
  $Sz(8)$&$2^6\cdot5\cdot7\cdot13$&$3$&$O_7(3)$&$2^9\cdot3^9\cdot5\cdot7\cdot13$&$2$\\
  $L_{2}(49)$    & $2^4\cdot3\cdot5^2\cdot7^2$ & $2^2$ &$U_6(2)$&$2^{15}\cdot3^6\cdot5\cdot7\cdot11$&$2\cdot3$\\
  $U_3(4)$&$2^6\cdot3\cdot5^2\cdot13$&$2^2$&$U_4(5)$&$2^5\cdot3^4\cdot5^4\cdot7\cdot13$&$2^2$\\
  $M_{12}$  &$2^{6}\cdot3^3\cdot5\cdot11$ & $2$&$A_{14}$&$2^{10}\cdot3^5\cdot5^2\cdot7^2\cdot11\cdot13$&$2$\\
  $U_3(5)$&$2^4\cdot3^2\cdot5^3\cdot7$&$2\cdot3$&$Suz$&$2^{13}\cdot3^7\cdot5^2\cdot7\cdot11\cdot13$&$2$\\
  $A_9$&$2^6\cdot3^4\cdot5\cdot7$&$2$&$Co_3$&$2^{10}\cdot3^7\cdot5^3\cdot7\cdot11\cdot23$&$1$\\
  $L_2(64)$&$2^6\cdot3^2\cdot5\cdot7\cdot13$&$2\cdot3$&$A_{15}$&$2^{10}\cdot3^6\cdot5^3\cdot7^2\cdot11\cdot13$&$2$\\
  $M_{22}$   &$2^7\cdot3^2\cdot5\cdot7\cdot11$ & $2$&$A_{16}$&$2^{14}\cdot3^6\cdot5^3\cdot7^2\cdot11\cdot13$&$2$\\
  $J_2$&$2^7\cdot3^3\cdot5^2\cdot7$&$2$&$Co_2$&$2^{18}\cdot 3^6\cdot5^3\cdot7\cdot11\cdot23$&$1$\\
  $S_6(2)$&$2^9\cdot3^4\cdot5\cdot7$&$1$&$Fi_{22}$&$2^{17}\cdot3^9\cdot5^2\cdot7\cdot11\cdot13$&$2$\\
  $A_{10}$&$2^7\cdot3^4\cdot5^2\cdot7$ &$2$&$Co_1$&$2^{21}\cdot 3^9\cdot5^4\cdot7^2\cdot11\cdot13\cdot23$&$1$\\
  $U_4(3)$&$2^7\cdot3^6\cdot5\cdot7$& $2^3$\\
\hline
\end{tabular}\end{table}

Suppose that $K/H$ is isomorphic to one of the groups in Table \ref{tb:Co1} except $Co_1$, $Co_{2}$, $Co_3$, $M_{24}$, $M_{23}$ and $L_2(23)$, then   $23\in\pi(H)$ and $|H_{23}|=23$. Also $H_{23}\unlhd G$. Since $|G/C_G(H_{23})|\big|22$, it follows that $\pi(G)\small\setminus\{11\}\subseteq\pi(C_G(H_{23}))$.  As  $\Gamma(G)$ is disconnected,   $11\not\in\pi(C_G(H_{23}))$.  By Lemma 2.4 and using the fact that $11\nmid|GL(2,7)|$, the action of an $11-$element $g$ on $C_G(H_{23})$ implies that $C_G(H_{23})$ has a $\left\langle g\right\rangle-$invariant $Sylow$ $7-$subgroup, so that  $11\cdot7\in\pi_e(G)$, which implies that $t(G)=1$, a contradiction.

 Suppose  that $K/H$ is isomorphic to one of  $Co_{2}$, $Co_3$, $M_{24}$, $M_{23}$ and $L_2(23)$, we deduce that $7\in\pi(H)$ and $|H_7|=7$ or $7^2$, $H_7\unlhd G$. By $|G/C_G(H_{7})|\big||GL(2,7)|$ and $|GL(2,7)|=2^5\cdot3^2\cdot7$, we have $\pi(G)=\pi(C_G(H_7))$, hence, $t(G)=1$, a contradiction.

If $K/H\cong Co_{1}$, then  $G\cong S$ by $|G|=|K/H|$.

\textbf{Step 2} If $S$ is one of $Co_{2}$, $Co_3$ and  $Fi_{23}$, then $G\cong S$.

Since $|G|$ equals one of
$|Co_{2}|=2^{18}\cdot 3^6\cdot5^3\cdot7\cdot11\cdot23$,
$|Co_3|=2^{10}\cdot 3^7\cdot5^3\cdot7\cdot11\cdot23$ and
$|Fi_{23}|=2^{18}\cdot 3^{13}\cdot5^2\cdot7\cdot11\cdot13\cdot17\cdot23$,
 by the same reason as above we get that $K/H$ is  one of simple groups in Table \ref{tb:Co2}-\ref{tb:Fi23} respectively.

 \begin{table}[H]
\setlength{\abovecaptionskip}{10pt}
\setlength{\belowcaptionskip}{0pt}
\centering
\caption{The non-abelian simple groups $K/H$ of orders dividing $|C o_2|$~}\label{tb:Co2}
\begin{tabular}{ccc|ccc}
 \hline
  $K/H$       & $|K/H|$
  & $|\mathrm{Out}(K/H)|$& $K/H$       & $|K/H|$
  & $|\mathrm{Out}(K/H)|$     \\
  \hline
  $A_5$       & $2^{2}\cdot3\cdot5$ & $2$& $A_9$&$2^6\cdot3^4\cdot5\cdot7$&$2$ \\
  $L_2(7)$ & $2^3\cdot3\cdot7$&$2$&$M_{22}$   &$2^7\cdot3^2\cdot5\cdot7\cdot11$&$2$\\
  $A_6$       & $2^{3}\cdot3^2\cdot5$ & $4$& $J_2$&$2^7\cdot3^3\cdot5^2\cdot7$&$2$\\
  $L_2(8)$& $2^{3}\cdot3^2\cdot7$ & $3$  &$S_6(2)$&$2^9\cdot3^4\cdot5\cdot7$&$1$\\
  $L_2(11)$& $2^{2}\cdot3\cdot5\cdot11$&$2$&$A_{10}$&$2^7\cdot3^4\cdot5^2\cdot7$&$2$\\
  $A_7$& $2^{3}\cdot3^2\cdot5\cdot7$&$2$&$U_4(3)$&$2^7\cdot3^6\cdot5\cdot7$&$2^3$\\
  $U_3(3)$&$2^{5}\cdot3^3\cdot7$&$2$&$M_{23}$&$2^7\cdot3^2\cdot5\cdot7\cdot11\cdot23$&$1$\\
  $L_2(23)$&$2^{3}\cdot3\cdot11\cdot23$&$2$&$U_5(2)$&$2^{10}\cdot3^5\cdot5\cdot11$&$2$\\ $M_{11}$&$2^4\cdot3^2\cdot5\cdot11$&$1$&$A_{11}$&$2^7\cdot3^4\cdot5^2\cdot7\cdot11$&$2$\\ $A_8$& $2^{6}\cdot3^2\cdot5\cdot7$&$2$&$HS$&$2^9\cdot3^2\cdot5^3\cdot7\cdot11$&$2$\\
  $L_3(4)$&$2^{6}\cdot3^2\cdot5\cdot7$&$2^2\cdot3$&$O_8^{+}(2)$&$2^{12}\cdot3^5\cdot5^2\cdot7$&$2\cdot3$\\ $U_{4}(2)$&$2^2\cdot3^4\cdot5$&$2$&$A_{12}$&$2^9\cdot3^5\cdot5^2\cdot7\cdot11$&$2$\\
  $M_{12}$&$2^{6}\cdot3^3\cdot5\cdot11$&$2$&$M_{24}$&$2^{10}\cdot3^3\cdot5\cdot7\cdot11\cdot23$&$1$\\
  $U_3(5)$&$2^4\cdot3^2\cdot5^3\cdot7$&$2\cdot3$&$M^cL$&$2^7\cdot3^6\cdot5^3\cdot7\cdot11$&$2$\\
  $U_6(2)$&$2^{15}\cdot3^6\cdot5\cdot7\cdot11$&$2\cdot3$&$Co_2$&$2^{18}\cdot 3^6\cdot5^3\cdot7\cdot11\cdot23$&$1$\\
  \hline
  \end{tabular}\end{table}

\begin{table}[H]
\setlength{\abovecaptionskip}{10pt}
\setlength{\belowcaptionskip}{0pt}
\centering
\caption{The non-abelian simple groups $K/H$ of orders dividing $|Co_3|$~}\label{tb:Co3}
\setlength{\tabcolsep}{4mm}{
\begin{tabular}{ccc|ccc}
 \hline
 $K/H$       & $|K/H|$
 & $|\mathrm{Out}(K/H)|$& $K/H$       & $|K/H|$
 & $|\mathrm{Out}(K/H)|$     \\
  \hline
  $A_5$       & $2^{2}\cdot3\cdot5$ & $2$& $A_9$&$2^6\cdot3^4\cdot5\cdot7$&$2$ \\
  $L_2(7)$ & $2^3\cdot3\cdot7$&$2$&$M_{22}$   &$2^7\cdot3^2\cdot5\cdot7\cdot11$&$2$\\
  $A_6$       & $2^{3}\cdot3^2\cdot5$ & $4$& $J_2$&$2^7\cdot3^3\cdot5^2\cdot7$&$2$\\
  $L_2(8)$& $2^{3}\cdot3^2\cdot7$ & $3$  &$S_6(2)$&$2^9\cdot3^4\cdot5\cdot7$&$1$\\
  $L_2(11)$& $2^{2}\cdot3\cdot5\cdot11$&$2$&$A_{10}$&$2^7\cdot3^4\cdot5^2\cdot7$&$2$\\
  $A_7$& $2^{3}\cdot3^2\cdot5\cdot7$&$2$&$U_4(3)$&$2^7\cdot3^6\cdot5\cdot7$&$2^3$\\
  $U_3(3)$&$2^{5}\cdot3^3\cdot7$&$2$&$M_{23}$&$2^7\cdot3^2\cdot5\cdot7\cdot11\cdot23$&$1$\\
  $L_2(23)$&$2^{3}\cdot3\cdot11\cdot23$&$2$&$U_5(2)$&$2^{10}\cdot3^5\cdot5\cdot11$&$2$\\ $M_{11}$&$2^4\cdot3^2\cdot5\cdot11$&$1$&$A_{11}$&$2^7\cdot3^4\cdot5^2\cdot7\cdot11$&$2$\\ $A_8$& $2^{6}\cdot3^2\cdot5\cdot7$&$2$&$HS$&$2^9\cdot3^2\cdot5^3\cdot7\cdot11$&$2$\\
  $L_3(4)$&$2^{6}\cdot3^2\cdot5\cdot7$&$2^2\cdot3$&$A_{12}$&$2^9\cdot3^5\cdot5^2\cdot7\cdot11$&$2$\\ $U_{4}(2)$&$2^2\cdot3^4\cdot5$&$2$&$M_{24}$&$2^{10}\cdot3^3\cdot5\cdot7\cdot11\cdot23$&$1$\\
  $M_{12}$&$2^{6}\cdot3^3\cdot5\cdot11$&$2$&$M^cL$&$2^7\cdot3^6\cdot5^3\cdot7\cdot11$&$2$\\
  $U_3(5)$&$2^4\cdot3^2\cdot5^3\cdot7$&$2\cdot3$&$Co_3$&$2^{10}\cdot3^7\cdot5^3\cdot7\cdot11\cdot23$&$1$\\
  \hline
  \end{tabular}}\end{table}

  \begin{longtabu}{p{1.5cm}p{2.6cm}p{1.9cm}|p{1.0cm}p{4.5cm}p{1.9cm}}

\caption{ The non-abelian simple groups $K/H$ of orders dividing $|Fi_{23}|$ }~\label{tb:Fi23}\\

\hline

 $K/H$       & $|K/H|$
 & $|\mathrm{Out}(K/H)|$&   $K/H$       & $|K/H|$
 & $|\mathrm{Out}(K/H)|$     \\
 \hline
 \endfirsthead
 \multicolumn{4}{l}{Table \ref{tb:Fi23}: (continued) ~} \\
 \hline
 $K/H$ & $|K/H|$
 & $|\mathrm{Out}(K/H)|$ & $K/H$  & $|K/H|$
 & $|\mathrm{Out}(K/H)|$ \\
 \hline
 \endhead
  $A_5$&$2^{2}\cdot3\cdot5$&$2$&$A_{10}$&$2^7\cdot3^4\cdot5^2\cdot7$&$2$\\
  $L_2(7)$ & $2^3\cdot3\cdot7$&$2$&$U_4(3)$&$2^7\cdot3^6\cdot5\cdot7$& $2^3$\\
  $A_6$&$2^{3}\cdot3^2\cdot5$&$4$&$G_2(3)$&$2^6\cdot3^6\cdot7\cdot13$&$2$\\
  $L_2(8)$&$2^{3}\cdot3^2\cdot7$&$3$&$L_4(3)$&$2^7\cdot3^6\cdot5\cdot13$&$2^2$\\
  $L_2(11)$&$2^{2}\cdot3\cdot5\cdot11$&$2$&$M_{23}$&$2^7\cdot3^2\cdot5\cdot7\cdot11\cdot23$&$1$\\ $L_2{13}$&$2^2\cdot3\cdot7\cdot13$&$2$&$U_5(2)$&$2^{10}\cdot3^5\cdot5\cdot11$ &$2$\\
  $L_2(17)$&$2^4\cdot3^2\cdot17$&$2$&$^2F_4(2)^{\prime}$&$2^{11}\cdot3^3\cdot5^2\cdot13$&$2$\\
  $A_7$&$2^{3}\cdot3^2\cdot5\cdot7$&$2$&$A_{11}$&$2^7\cdot3^4\cdot5^2\cdot7\cdot11$&$2$\\
  $L_2(16)$&$2^4\cdot3\cdot5\cdot17$&$2^2$&$L_3(9)$&$2^7\cdot3^6\cdot5\cdot7\cdot13$&$2^2$\\
  $L_3(3)$&$2^4\cdot3^3\cdot13$&$2$&$O_8^{+}(2)$&$2^{12}\cdot3^5\cdot5^2\cdot7$&$2\cdot3$\\
  $U_3(3)$&$2^{5}\cdot3^3\cdot7$&$2$&$O_8^{-}(2)$&$2^{12}\cdot3^4\cdot5\cdot7\cdot17$&$2$\\
  $L_2(23)$&$2^{3}\cdot3\cdot11\cdot23$&$2$&$S_6(2)$&$2^9\cdot3^4\cdot5\cdot7$&$1$\\
  $L_2(25)$&$2^3\cdot3\cdot5^2\cdot13$&$2^2$&$A_{12}$&$2^9\cdot3^5\cdot5^2\cdot7\cdot11$&$2$\\
  $M_{11}$&$2^4\cdot3^2\cdot5\cdot11$&$1$&$M_{24}$&$2^{10}\cdot3^3\cdot5\cdot7\cdot11\cdot23$&$1$\\
  $L_2(27)$&$2^2\cdot3^3\cdot7\cdot13$&$2\cdot3$&$G_2(4)$&$2^{12}\cdot3^3\cdot5^2\cdot7\cdot13$&$2$\\
  $A_8$&$2^{6}\cdot3^2\cdot5\cdot7$&$2$&$L_4(4)$&$2^{12}\cdot3^4\cdot5^2\cdot7\cdot17$&$2^2$\\
  $L_3(4)$&$2^{6}\cdot3^2\cdot5\cdot7$&$2^2\cdot3$&$L_3(16)$&$2^{12}\cdot3^2\cdot5^2\cdot7\cdot13\cdot17$&$2^3\cdot3$\\ $U_{4}(2)$&$2^2\cdot3^4\cdot5$&$2$&$A_{13}$&$2^{9}\cdot3^5\cdot5^2\cdot7\cdot11\cdot13$&$2$\\ $Sz(8)$&$2^6\cdot5\cdot7\cdot13$&$3$&$O_7(3)$&$2^9\cdot3^9\cdot5\cdot7\cdot13$&$2$\\
  $M_{12}$&$2^{6}\cdot3^3\cdot5\cdot11$&$2$&$S_6(3)$&$2^9\cdot3^9\cdot5\cdot7\cdot13$&$2$\\
  $U_3(4)$&$2^6\cdot3\cdot5^2\cdot13$&$2^2$&$U_6(2)$&$2^{15}\cdot3^6\cdot5\cdot7\cdot11$&$2\cdot3$\\
  $A_9$&$2^6\cdot3^4\cdot5\cdot7$&$2$&$S_8(2)$&$2^{16}\cdot3^5\cdot5^2\cdot7\cdot17$&$1$\\
  $L_2(64)$&$2^6\cdot3^2\cdot5\cdot7\cdot13$&$2\cdot3$&$Suz$&$2^{13}\cdot3^7\cdot5^2\cdot7\cdot11\cdot13$&$2$\\
  $M_{22}$&$2^7\cdot3^2\cdot5\cdot7\cdot11$&$2$&$O_8^{+}(3)$&$2^{12}\cdot3^{12}\cdot5^2\cdot7\cdot13$&$2^3\cdot3$\\
  $J_2$&$2^7\cdot3^3\cdot5^2\cdot7$&$2$&$Fi_{22}$&$2^{17}\cdot3^9\cdot5^2\cdot7\cdot11\cdot13$&$2$\\
  $S_4(4)$&$2^8\cdot3^2\cdot5^2\cdot17$&$2^2$&$Fi_{23}$&$2^{18}\cdot3^{13}\cdot5^2\cdot7\cdot11\cdot13\cdot17\cdot23$&$1$\\

  \hline
   \end{longtabu}

  From Table \ref{tb:Co2}-\ref{tb:Fi23},  we have

  (1) If  $23\nmid|K/H|$, then $23\nmid|\rm{Out(K/H)}|$.

 (2) If  $23\big||K/H|$, then $K/H$ is isomorphic to one of: $S$, $ L_2(23)$, $M_{23}$ and $M_{24}$. Moreover, $5\nmid|\rm{Out(K/H)}|$.

   For $K/H$ satisfying (1), with the same arguments as  in \textbf{Step 1}, we can get contradictions. For $K/H$ satisfying  (2), if $K/H\cong S$, then  $G\cong S$, we are done. For the remain cases, observing that  $|L_2(23)|=2^{3}\cdot3\cdot11\cdot23$, $|M_{23}|=2^7\cdot3^2\cdot5\cdot7\cdot11\cdot23$, $|M_{24}|=2^{10}\cdot3^3\cdot5\cdot7\cdot11\cdot23$, and by (2) , we deduce that $5\in\pi(H)$, $|H_5|\big|5^3$ and  $H_5\unlhd G$. It forces that $\pi(G)=\pi(C_G(H_5))$ since $|G/C_G(H_5)|\big||GL(3,5)|$ and $|GL(3,5)|=2^7\cdot3\cdot5^3\cdot31$, hence, $t(G)=1$, a contradiction.

\end{proof}

\begin{lemma} \label{le4} Let $G$ be a group and   $S=Fi_{22}$ or $Fi^{\prime}_{24}$.   If $|G|=|S|$ and the prime graph of $G$ is disconnected, then $G\cong S$.\end{lemma}

\begin{proof} The proof is divided into following steps.

\textbf{Step 1} If $S= Fi_{22}$, then  $G\cong Fi_{22}$.

In this case, $|G|=|Fi_{22}|=2^{17}\cdot 3^{9}\cdot5^2\cdot7\cdot11\cdot13$. By the same reason, checking the list in \cite{ATLAS}, we get that $K/H$ is isomorphic to one of the groups in Table \ref{tb:Fi22}.

\begin{table}[h]
\setlength{\abovecaptionskip}{10pt}
\setlength{\belowcaptionskip}{0pt}
\centering
\caption{The non-abelian simple groups with orders dividing $|Fi_{22}|$ }~\label{tb:Fi22}
\setlength{\tabcolsep}{4mm}{
\begin{tabular}{ccc|ccc}
 \hline
  $K/H$      & $|K/H|$
  & $|\mathrm{Out}(K/H)|$& $K/H$      & $|K/H|$
  & $|\mathrm{Out}(K/H)|$     \\
  \hline
  $A_5$&$2^{2}\cdot3\cdot5$&$2$&$A_{10}$&$2^7\cdot3^4\cdot5^2\cdot7$&$2$\\
  $L_2(7)$ & $2^3\cdot3\cdot7$&$2$&$U_4(3)$&$2^7\cdot3^6\cdot5\cdot7$& $2^3$\\
  $A_6$&$2^{3}\cdot3^2\cdot5$&$4$&$G_2(3)$&$2^6\cdot3^6\cdot7\cdot13$&$2$\\
  $L_2(8)$&$2^{3}\cdot3^2\cdot7$&$3$&$L_4(3)$&$2^7\cdot3^6\cdot5\cdot13$&$2^2$\\
  $L_2(11)$&$2^{2}\cdot3\cdot5\cdot11$&$2$&$^2F_4(2)^{\prime}$&$2^{11}\cdot3^3\cdot5^2\cdot13$&$2$\\ $L_2{(13)}$&$2^2\cdot3\cdot7\cdot13$&$2$&$U_5(2)$&$2^{10}\cdot3^5\cdot5\cdot11$&$2$\\
  $A_7$&$2^{3}\cdot3^2\cdot5\cdot7$&$2$&$A_{11}$&$2^7\cdot3^4\cdot5^2\cdot7\cdot11$&$2$\\
  $L_3(9)$&$2^7\cdot3^6\cdot5\cdot7\cdot13$&$2^2$&$O_8^{+}(2)$&$2^{12}\cdot3^5\cdot5^2\cdot7$&$2\cdot3$\\
  $L_3(3)$&$2^4\cdot3^3\cdot13$&$2$&$S_6(2)$&$2^9\cdot3^4\cdot5\cdot7$&$1$\\
  $U_3(3)$&$2^{5}\cdot3^3\cdot7$&$2$&$A_{12}$&$2^9\cdot3^5\cdot5^2\cdot7\cdot11$&$2$\\
  $L_2(25)$&$2^3\cdot3\cdot5^2\cdot13$&$2^2$&$A_8$&$2^{6}\cdot3^2\cdot5\cdot7$&$2$\\
  $M_{11}$&$2^4\cdot3^2\cdot5\cdot11$&$1$&$G_2(4)$&$2^{12}\cdot3^3\cdot5^2\cdot7\cdot13$&$2$\\
  $L_2(27)$&$2^2\cdot3^3\cdot7\cdot13$&$2\cdot3$&$A_9$&$2^6\cdot3^4\cdot5\cdot7$&$2$\\
  $L_3(4)$&$2^{6}\cdot3^2\cdot5\cdot7$&$2^2\cdot3$&$A_{13}$&$2^{9}\cdot3^5\cdot5^2\cdot7\cdot11\cdot13$&$2$\\ $U_{4}(2)$&$2^2\cdot3^4\cdot5$&$2$&$O_7(3)$&$2^9\cdot3^9\cdot5\cdot7\cdot13$&$2$\\ $Sz(8)$&$2^6\cdot5\cdot7\cdot13$&$3$&$S_6(3)$&$2^9\cdot3^9\cdot5\cdot7\cdot13$&$2$\\
  $M_{12}$&$2^{6}\cdot3^3\cdot5\cdot11$&$2$&$U_6(2)$&$2^{15}\cdot3^6\cdot5\cdot7\cdot11$&$2\cdot3$\\
  $U_3(4)$&$2^6\cdot3\cdot5^2\cdot13$&$2^2$&$Suz$&$2^{13}\cdot3^7\cdot5^2\cdot7\cdot11\cdot13$&$2$\\
  $L_2(64)$&$2^6\cdot3^2\cdot5\cdot7\cdot13$&$2\cdot3$&$J_2$&$2^7\cdot3^3\cdot5^2\cdot7$&$2$\\
  $M_{22}$&$2^7\cdot3^2\cdot5\cdot7\cdot11$&$2$&$Fi_{22}$&$2^{17}\cdot3^9\cdot5^2\cdot7\cdot11\cdot13$&$2$\\
  \hline
  \end{tabular}}\end{table}

Suppose that $K/H$ is one of the groups in Table \ref{tb:Fi22} except $L_2(25)$, $U_{3}(4)$, $J_{2}$, $A_{10}$, $^2F_4(2)^{\prime}$, $O_8^{+}(2)$,  $G_2(4)$, $A_{11}$, $A_{12}$, $A_{13}$, $Suz$ and $Fi_{22}$, then $5\in\pi(H)$ and $|H_{5}|=5$ or $5^2$. Clearly $H_{5}\unlhd G$.  Further, we deduce  contradictions to $t(G)=1$ since $|G/C_G(H_{5})|\big||GL(2,5)|$ and $|GL(2,5)|=2^5\cdot3\cdot5$.

Suppose that $K/H$ is isomorphic to one of $L_2(25)$, $U_{3}(4)$, $J_{2}$, $A_{10}$, $^2F_4(2)^{\prime}$,  $O_8^{+}(2)$ and  $G_2(4)$, then $11\in\pi(H)$ and $|H_{11}|=11$, also $H_{11}\unlhd G$. Thereby, we deduce  contradictions  to $t(G)=1$ since $|G/C_G(H_{11})|\big|2\cdot5$.

Suppose that $K/H$ is isomorphic to  $A_{11}$ or  $A_{12}$, then $13\in\pi(H)$ and $|H_{13}|=13$, also $H_{13}\unlhd G$. Similarly, we get contradictions to $t(G)=1$ since $|G/C_G(H_{13})|\big|2^2\cdot3$.

Suppose that $K/H$ is isomorphic to  $A_{13}$ or  $Suz$, then $3\in\pi(H)$ and $|H_{3}|=3^2$ or $3^4$, also $H_{3}\unlhd G$.  Since $|G/C_G(H_{3})|\big||GL(4,3)|$ and $|GL(4,3)|=2^9\cdot3^6\cdot5\cdot13$, it follows that $\pi(G)\smallsetminus\{13\}\subseteq\pi(C_G(H_3))$. In view of the fact that  $\Gamma(G)$ is disconnected, we have  $13\not\in\pi(C_G(H_3))$. Now looking at the action of a $13-$element $g$ on $C_G(H_3)$ and applying  Lemma 2.4, we get that $C_G(H_3)$ has a $\left\langle g\right\rangle-$invariant $\rm{Sylow}\ 7-$subgroup. Therefore,  $13\cdot7\in\pi_e(G)$, which implies that $\Gamma(G)$ is connected, a contradiction.

Now we have $K/H\cong Fi_{22}$, then $G\cong S\cong Fi_{22}$ as desired.

\end{proof}

\textbf{Step 2} If $S= Fi^{\prime}_{24}$ then $G\cong Fi^{\prime}_{24}$.

 In this case, $|G|=|Fi^{\prime}_{24}|=2^{21}\cdot 3^{16}\cdot5^2\cdot7^3\cdot11\cdot13\cdot17\cdot23\cdot29$.
  By \cite{ATLAS} and  Lemma \ref{lemT}, we get that $K/H$ may be one of the groups in Table \ref{tb:Fi24}.

  \setlength{\LTleft}{0pt} \setlength{\LTright}{0pt} 
\setlength\LTleft{0in}
\setlength\LTright{+1in plus 2 fill}
\setlength{\tabcolsep}{3pt}
\begin{longtabu}{p{1.5cm}p{3.3cm}c|p{2.0cm}p{4.5cm}c}
\caption{the non-abelian simple groups $K/H$ of orders dividing $|Fi^{\prime}_{24}|$ }~\label{tb:Fi24}\\

\hline

  $K/H$       & $|K/H|$
  & {\small $|\mathrm{Out}(K/H)|$}&   $K/H$       & $|K/H|$
  & {\small $|\mathrm{Out}(K/H)|$}    \\
 \hline
 \endfirsthead
 \multicolumn{4}{l}{Table \ref{tb:Fi24}: (continued) ~} \\
 \hline
  $K/H$       & $|K/H|$
 & {\small $|\mathrm{Out}(K/H)|$}&   $K/H$       & $|K/H|$
 & {\small $|\mathrm{Out}(K/H)|$}     \\
 \hline
 \endhead

 \hline
  $A_5$&$2^{2}\cdot3\cdot5$&$2$&$A_{10}$&$2^7\cdot3^4\cdot5^2\cdot7$&$2$\\
  $L_2(7)$ & $2^3\cdot3\cdot7$&$2$&$U_4(3)$&$2^7\cdot3^6\cdot5\cdot7$& $2^3$\\
  $A_6$&$2^{3}\cdot3^2\cdot5$&$4$&$G_2(3)$&$2^6\cdot3^6\cdot7\cdot13$&$2$\\
  $L_2(8)$&$2^{3}\cdot3^2\cdot7$&$3$&$L_4(3)$&$2^7\cdot3^6\cdot5\cdot13$&$2^2$\\
  $L_2(11)$&$2^{2}\cdot3\cdot5\cdot11$&$2$&$M_{23}$&$2^7\cdot3^2\cdot5\cdot7\cdot11\cdot23$&$1$\\ $L_2{13}$&$2^2\cdot3\cdot7\cdot13$&$2$&$U_5(2)$&$2^{10}\cdot3^5\cdot5\cdot11$ &$2$\\
  $L_2(17)$&$2^4\cdot3^2\cdot17$&$2$&$^2F_4(2)^{\prime}$&$2^{11}\cdot3^3\cdot5^2\cdot13$&$2$\\
  $A_7$&$2^{3}\cdot3^2\cdot5\cdot7$&$2$&$A_{11}$&$2^7\cdot3^4\cdot5^2\cdot7\cdot11$&$2$\\
  $L_2(16)$&$2^4\cdot3\cdot5\cdot17$&$2^2$&$L_3(9)$&$2^7\cdot3^6\cdot5\cdot7\cdot13$&$2^2$\\
  $L_3(3)$&$2^4\cdot3^3\cdot13$&$2$&$O_8^{+}(2)$&$2^{12}\cdot3^5\cdot5^2\cdot7$&$2\cdot3$\\
  $U_3(3)$&$2^{5}\cdot3^3\cdot7$&$2$&$O_8^{-}(2)$&$2^{12}\cdot3^4\cdot5\cdot7\cdot17$&$2$\\
  $L_2(23)$&$2^{3}\cdot3\cdot11\cdot23$&$2$&$^3D_4{2}$&$2^{12}\cdot3^4\cdot7^2\cdot13$&$3$\\
  $L_2(25)$&$2^3\cdot3\cdot5^2\cdot13$&$2^2$&$A_{12}$&$2^9\cdot3^5\cdot5^2\cdot7\cdot11$&$2$\\
  $M_{11}$&$2^4\cdot3^2\cdot5\cdot11$&$1$&$M_{24}$&$2^{10}\cdot3^3\cdot5\cdot7\cdot11\cdot23$&$1$\\
  $L_2(27)$&$2^2\cdot3^3\cdot7\cdot13$&$2\cdot3$&$G_2(4)$&$2^{12}\cdot3^3\cdot5^2\cdot7\cdot13$&$2$\\
  $L_2(29)$&$2^2\cdot3\cdot5\cdot7\cdot29$&$2$&$L_4(4)$&$2^{12}\cdot3^4\cdot5^2\cdot7\cdot17$&$2^2$\\
  $A_8$&$2^{6}\cdot3^2\cdot5\cdot7$&$2$&$S_4(8)$&$2^{12}\cdot3^4\cdot5\cdot7^2\cdot13$&$2\cdot3$\\
  $L_3(4)$&$2^{6}\cdot3^2\cdot5\cdot7$&$2^2\cdot3$&$L_3(16)$&$2^{12}\cdot3^2\cdot5^2\cdot7\cdot13\cdot17$&$2^3\cdot3$\\ $U_{4}(2)$&$2^2\cdot3^4\cdot5$&$2$&$A_{13}$&$2^{9}\cdot3^5\cdot5^2\cdot7\cdot11\cdot13$&$2$\\ $Sz(8)$&$2^6\cdot5\cdot7\cdot13$&$3$&$O_7(3)$&$2^9\cdot3^9\cdot5\cdot7\cdot13$&$2$\\
  $L_{2}(49)$&$2^4\cdot3\cdot5^2\cdot7^2$&$2^2$&$S_6(3)$&$2^9\cdot3^9\cdot5\cdot7\cdot13$&$2$\\
  $U_3(4)$&$2^6\cdot3\cdot5^2\cdot13$&$2^2$&$U_6(2)$&$2^{15}\cdot3^6\cdot5\cdot7\cdot11$&$2\cdot3$\\
  $M_{12}$&$2^{6}\cdot3^3\cdot5\cdot11$&$2$&$A_{14}$&$2^{10}\cdot3^5\cdot5^2\cdot7^2\cdot11\cdot13$&$2$\\
  $A_9$&$2^6\cdot3^4\cdot5\cdot7$&$2$&$S_8(2)$&$2^{16}\cdot3^5\cdot5^2\cdot7\cdot17$&$1$\\
  $L_2(64)$&$2^6\cdot3^2\cdot5\cdot7\cdot13$&$2\cdot3$&$Suz$&$2^{13}\cdot3^7\cdot5^2\cdot7\cdot11\cdot13$&$2$\\
  $M_{22}$&$2^7\cdot3^2\cdot5\cdot7\cdot11$&$2$&$O_8^{+}(3)$&$2^{12}\cdot3^{12}\cdot5^2\cdot7\cdot13$&$2^3\cdot3$\\
  $J_2$&$2^7\cdot3^3\cdot5^2\cdot7$&$2$&$Fi_{22}$&$2^{17}\cdot3^9\cdot5^2\cdot7\cdot11\cdot13$&$2$\\
  $S_4(4)$&$2^8\cdot3^2\cdot5^2\cdot17$&$2^2$&$Fi_{23}$&$2^{18}\cdot3^{13}\cdot5^2\cdot7\cdot11\cdot13\cdot17\cdot23$&$1$\\
  $S_6(2)$&$2^9\cdot3^4\cdot5\cdot7$&$1$&$Fi^{\prime}_{24}$&$2^{21}\cdot3^{16}\cdot5^2\cdot7^3\cdot11\cdot13\cdot17\cdot23\cdot29$&$2$\\
  \hline
  \end{longtabu}

Suppose that $K/H$ is isomorphic to one of the groups in Table \ref{tb:Fi24} except $Fi^{\prime}_{24}$ or $L_2(29)$, then   $29\in\pi(H)$. Of course,  $|H_{29}|=29$ and $H_{29}\unlhd G$. It follows from $|G/C_G(H_{29})|\big|2^2\cdot7$ that $t(G)=1$, a contradiction.

 Suppose  that    $K/H\cong Fi^{\prime}_{24}$, then $G\cong S\cong Fi^{\prime}_{24}$, we are done. Now suppose that $K/H\cong L_2(29)$. Observe that  $|L_2(29)|=2^2\cdot3\cdot5\cdot7\cdot29$, and $|\rm{Out}(L_2(29))|=2$, we deduce that $11\in\pi(H)$,  $|H_{11}|=11$ and $H_{11}\unlhd G$. Then  by  $|G/C_G(H_{11})|\big|2\cdot5$, we get $t(G)=1$, a contradiction.

\begin{lemma}\label{le5} Let $G$ be a group and  $S$ a Monster group or Baby Monster group.     If $|G|=|S|$ and the prime graph of $G$ is disconnected, then $G\cong S$.\end{lemma}

\begin{proof} We have the following two steps.

\textbf{Step 1} If $S= M$, then $G\cong M$.

 Here, $|G|=|M|=2^{46}\cdot3^{20}\cdot5^9\cdot7^6\cdot11^2\cdot13^3\cdot17\cdot19\cdot23\cdot29\cdot31\cdot41\cdot47\cdot59\cdot71$.
 By Lemma \ref{lemT}, $K/H$ may be a simple group in the list of \cite{ATLAS} and in Lemma \ref{lemT}, we list all of them  in Table \ref{tb:M}.

 \setlength{\LTleft}{0pt} \setlength{\LTright}{0pt} 
\setlength\LTleft{0in}
\setlength\LTright{+1in plus 2 fill}
\setlength{\tabcolsep}{3pt}
{\tiny
\begin{longtabu}{p{1.5cm}p{4.5cm}c|p{1.2cm}p{4.8cm}c}
\caption{the non-abelian simple groups $K/H$ of orders dividing $|M|$ }~\label{tb:M}\\

\hline

  $K/H$       & $|K/H|$
  &  $|\mathrm{Out}(K/H)|$&   $K/H$       & $|K/H|$
  &  $|\mathrm{Out}(K/H)|$    \\
 \hline
 \endfirsthead
 \multicolumn{4}{l}{Table \ref{tb:M}: (continued) ~} \\
 \hline
  $K/H$       & $|K/H|$
 &  $|\mathrm{Out}(K/H)|$&   $K/H$       & $|K/H|$
 &  $|\mathrm{Out}(K/H)|$     \\
 \hline
 \endhead

\hline
 $A_5$&$2^{2}\cdot3\cdot5$&$2$&$Sz(32)$&$2^{10}\cdot5^2\cdot31\cdot41$&$5$\\
  $L_2(7)$&$2^3\cdot3\cdot7$&$2$&$HS$&$2^{9}\cdot3^2\cdot5^3\cdot7\cdot11$&$2$\\
  $A_6$&$2^{3}\cdot3^2\cdot5$&$4$&$J_3$&$2^{7}\cdot3^5\cdot5\cdot17\cdot19$&$2$\\
  $L_2(8)$&$2^{3}\cdot3^2\cdot7$&$3$&$S_4(7)$&$2^{8}\cdot3^2\cdot5^2\cdot7^4$&$2$\\
  $L_2(11)$&$2^{2}\cdot3\cdot5\cdot11$&$2$&$O_8^{+}(2)$&$2^{12}\cdot3^5\cdot5^2\cdot7$&$2\cdot3$\\ $L_2{(13)}$&$2^2\cdot3\cdot7\cdot13$&$2$&$O_8^{-}(2)$&$2^{12}\cdot3^4\cdot5\cdot7\cdot17$&$2$\\
  $L_2(17)$&$2^4\cdot3^2\cdot17$&$2$&$^3D_4{(2)}$&$2^{12}\cdot3^4\cdot7^2\cdot13$&$3$\\
  $A_7$&$2^{3}\cdot3^2\cdot5\cdot7$&$2$&$A_{12}$&$2^9\cdot3^5\cdot5^2\cdot7\cdot11$&$2$\\
  $L_2(19)$&$2^2\cdot3^2\cdot5\cdot19$&$2$&$M_{24}$&$2^{10}\cdot3^3\cdot5\cdot7\cdot11\cdot23$&$1$\\
  $L_2(16)$&$2^4\cdot3\cdot5\cdot17$&$2^2$&$G_2(4)$&$2^{12}\cdot3^3\cdot5^2\cdot7\cdot13$&$2$\\
  $L_3(3)$&$2^4\cdot3^3\cdot13$&$2$&$M^cL$&$2^{7}\cdot3^6\cdot5^3\cdot7\cdot11$&$2$\\
  $U_3(3)$&$2^{5}\cdot3^3\cdot7$&$2$&$L_3(9)$&$2^7\cdot3^6\cdot5\cdot7\cdot13$&$2^2$\\
  $L_2(23)$&$2^{3}\cdot3\cdot11\cdot23$&$2$&$L_4(4)$&$2^{12}\cdot3^4\cdot5^2\cdot7\cdot17$&$2^2$\\
  $L_2(25)$&$2^3\cdot3\cdot5^2\cdot13$&$2^2$&$U_4(4)$&$2^{12}\cdot3^2\cdot5^3\cdot13\cdot17$&$2^2$\\
  $M_{11}$&$2^4\cdot3^2\cdot5\cdot11$&$1$&$S_4(8)$&$2^{12}\cdot3^4\cdot5\cdot7^2\cdot13$&$2\cdot3$\\
  $L_2(27)$&$2^2\cdot3^3\cdot7\cdot13$&$2\cdot3$&$L_3(16)$&$2^{12}\cdot3^2\cdot5^2\cdot7\cdot13\cdot17$&$2^3\cdot3$\\
  $L_2(29)$&$2^2\cdot3\cdot5\cdot7\cdot29$&$2$&$S_4(9)$&$2^{8}\cdot3^8\cdot5^2\cdot41$&$2^2$\\
  $L_2(31)$&$2^5\cdot3\cdot5\cdot31$&$2$&$A_{13}$&$2^{9}\cdot3^5\cdot5^2\cdot7\cdot11\cdot13$&$2$\\
  $A_8$&$2^{6}\cdot3^2\cdot5\cdot7$&$2$&$He$&$2^{10}\cdot3^3\cdot5^2\cdot7^3\cdot17$&$2$\\
  $L_3(4)$&$2^{6}\cdot3^2\cdot5\cdot7$&$2^2\cdot3$&$S_6(3)$&$2^{9}\cdot3^9\cdot5\cdot7\cdot7\cdot13$&$2$\\ $U_{4}(2)$&$2^2\cdot3^4\cdot5$&$2$&$O_7(3)$&$2^9\cdot3^9\cdot5\cdot7\cdot13$&$2$\\ $Sz(8)$&$2^6\cdot5\cdot7\cdot13$&$3$&$G_2(5)$&$2^{6}\cdot3^3\cdot5^6\cdot7\cdot31$&$1$\\
  $L_2(32)$&$2^5\cdot3\cdot11\cdot31$&$5$&$L_4(5)$&$2^{7}\cdot3^2\cdot5^6\cdot13\cdot31$&$8$\\
  $L_2(41)$&$2^3\cdot3\cdot5\cdot7\cdot41$&$2$&$U_6(2)$&$2^{15}\cdot3^6\cdot5\cdot7\cdot11$&$2\cdot3$\\
  $L_2(47)$&$2^4\cdot3\cdot23\cdot47$&$2$&$U_4(5)$&$2^{5}\cdot3^4\cdot5^4\cdot7\cdot13$&$2^2$\\
  $L_{2}(49)$&$2^4\cdot3\cdot5^2\cdot7^2$&$2^2$&$L_6(2)$&$2^{15}\cdot3^4\cdot5\cdot7^2\cdot31$&$2$\\
  $U_3(4)$&$2^6\cdot3\cdot5^2\cdot13$&$2^2$&$A_{14}$&$2^{10}\cdot3^5\cdot5^2\cdot7^2\cdot11\cdot13$&$2$\\
  $M_{12}$&$2^{6}\cdot3^3\cdot5\cdot11$&$2$&$S_{8}(2)$&$2^{16}\cdot3^5\cdot5^2\cdot7\cdot17$&$1$\\
  $L_2(59)$&$2^2\cdot3\cdot5\cdot29\cdot59$&$2$&$L_{3}(25)$&$2^{7}\cdot3^2\cdot5^6\cdot7\cdot13\cdot31$&$12$\\
  $U_3(5)$&$2^4\cdot3^2\cdot5^3\cdot7$&$6$&$Ru$&$2^{14}\cdot3^3\cdot5^3\cdot7\cdot13\cdot29$&$1$\\
  $J_1$&$2^3\cdot3\cdot5\cdot7\cdot11\cdot19$&$1$&$L_{5}(3)$&$2^{9}\cdot3^{10}\cdot5\cdot11^2\cdot13$&$2$\\
  $L_2(71)$&$2^3\cdot3^2\cdot5\cdot7\cdot71$&$2$&$Suz$&$2^{13}\cdot3^7\cdot5^2\cdot7\cdot11\cdot13$&$2$\\
  $A_9$&$2^6\cdot3^4\cdot5\cdot7$&$2$&$O^{\prime}N$&$2^{9}\cdot3^4\cdot5\cdot7^3\cdot11\cdot19\cdot31$&$2$\\
  $L_2(64)$&$2^6\cdot3^2\cdot5\cdot7\cdot13$&$2\cdot3$&$Co_3$&$2^{10}\cdot3^7\cdot5^3\cdot7\cdot11\cdot23$&$1$\\
  $L_2(81)$&$2^4\cdot3^4\cdot5\cdot41$&$2^3$&$A_{15}$&$2^{10}\cdot3^6\cdot5^2\cdot7^6\cdot19\cdot13$&$4$\\
  $L_3(5)$&$2^5\cdot3\cdot5^3\cdot31$&$2$&$L_4(7)$&$2^{9}\cdot3^4\cdot5^2\cdot7\cdot11\cdot13$&$2$\\
  $M_{22}$&$2^7\cdot3^2\cdot5\cdot7\cdot11$&$2$&$S_6(4)$&$2^{18}\cdot3^4\cdot5^3\cdot7\cdot13\cdot17$&$2$\\
  $J_2$&$2^7\cdot3^3\cdot5^2\cdot7$&$2$&$O_8^{+}(3)$&$2^{12}\cdot3^{12}\cdot5^2\cdot7\cdot13$&$2^3\cdot3$\\
  $L_2(125)$&$2^2\cdot3^2\cdot5^3\cdot7\cdot31$&$2\cdot3$&$O_8^{-}(3)$&$2^{10}\cdot3^{12}\cdot5\cdot7\cdot13\cdot41$&$2^2$\\
  $S_4(4)$&$2^8\cdot3^2\cdot5^2\cdot17$&$2^2$&$A_{16}$&$2^{14}\cdot3^{6}\cdot5^3\cdot7^2\cdot11\cdot13$&$2$\\
  $S_6(2)$&$2^9\cdot3^4\cdot5\cdot7$&$1$&$O_{10}^{+}(2)$&$2^{20}\cdot3^{5}\cdot5^2\cdot7\cdot17\cdot31$&$2$\\
  $A_{10}$&$2^7\cdot3^4\cdot5^2\cdot7$&$2$&$O_{10}^{-}(2)$&$2^{20}\cdot3^{6}\cdot5^2\cdot7\cdot11\cdot17$&$2$\\
  $L_3(7)$&$2^5\cdot3^2\cdot7^3\cdot19$&$2\cdot3$&$U_4(8)$&$2^{18}\cdot3^7\cdot5\cdot7^2\cdot13\cdot19$&$6$\\
  $U_4(3)$&$2^7\cdot3^6\cdot5\cdot7$&$2^3$&$Co_2$&$2^{18}\cdot3^6\cdot5^3\cdot7\cdot11\cdot23$&$1$\\
  $G_2(3)$&$2^6\cdot3^6\cdot7\cdot13$&$2$&$L_4(9)$&$2^{20}\cdot3^{2}\cdot5^4\cdot13\cdot17\cdot41$&$20$\\
  $S_4(5)$&$2^6\cdot3^2\cdot5^4\cdot13$&$2$&$U_5(4)$&$2^{18}\cdot3^7\cdot5\cdot7^2\cdot13\cdot19$&$6$\\
  $U_3(8)$&$2^9\cdot3^4\cdot7\cdot19$&$2$&$Fi_{22}$&$2^{17}\cdot3^9\cdot5^2\cdot7\cdot11\cdot13$&$2$\\
  $L_4(3)$&$2^7\cdot3^6\cdot5\cdot13$&$2^2$&$A_{17}$&$2^{14}\cdot3^{6}\cdot5^3\cdot7^2\cdot11\cdot13\cdot17$&$2$\\
  $L_5(2)$&$2^{10}\cdot3^2\cdot5\cdot7\cdot31$&$2$&$S_6(5)$&$2^{9}\cdot3^{4}\cdot5^9\cdot7\cdot13\cdot31$&$2$\\
  $M_{23}$&$2^7\cdot3^2\cdot5\cdot7\cdot11\cdot23$&$1$&$O_7(5)$&$2^{9}\cdot3^{4}\cdot5^9\cdot7\cdot13\cdot31$&$2$\\
  $U_5(2)$&$2^{10}\cdot3^5\cdot5\cdot11$&$2$&$L_5(4)$&$2^{20}\cdot3^5\cdot5^2\cdot7\cdot11\cdot17\cdot31$&$2^2$\\
  $^2F_4(2)^{\prime}$&$2^{11}\cdot3^3\cdot5^2\cdot13$&$2$&$HN$&$2^{14}\cdot3^6\cdot5^6\cdot7\cdot11\cdot19$&$2$\\
  $A_{11}$&$2^7\cdot3^4\cdot5^2\cdot7\cdot11$&$2$&$A_{18}$&$2^{15}\cdot3^8\cdot5^3\cdot7^2\cdot11\cdot13\cdot17$&$2$\\
  $F_4(2)$&$2^{24}\cdot3^6\cdot5^2\cdot7^2\cdot13\cdot17$&$2$&$L_6(4)$&$2^{30}\cdot3^{6}\cdot5^3\cdot7^2\cdot11\cdot13\cdot17\cdot31$&$12$\\
  $L_6(3)$&$2^{11}\cdot3^{15}\cdot5\cdot7\cdot11^2\cdot13^2$&$2^2$&$A_{22}$&$2^{18}\cdot3^9\cdot5^4\cdot7^3\cdot11^2\cdot13\cdot17\cdot19$&$2$\\
  $S_{10}(2)$&$2^{25}\cdot3^6\cdot5^2\cdot7\cdot11\cdot17\cdot31$&$1$&$U_{6}(4)$&$2^{30}\cdot3^4\cdot5^6\cdot7\cdot13^2\cdot\cdot17\cdot41$&$4$\\
  $A_{19}$&$2^{15}\cdot3^8\cdot5^3\cdot7^2\cdot11\cdot13\cdot17\cdot19$&$2$&$O^{+}_{10}(3)$&$2^{15}\cdot3^{20}\cdot5^2\cdot7\cdot11^2\cdot13\cdot41$&$2^2$\\
  $S_{8}(3)$&$2^{14}\cdot3^{16}\cdot5^2\cdot7\cdot13\cdot41$&$2$&$A_{23}$&$2^{18}\cdot3^9\cdot5^4\cdot7^3\cdot11^2\cdot13\cdot17\cdot19\cdot23$&$2$\\
  $O_{9}(3)$&$2^{14}\cdot3^{16}\cdot5^2\cdot7\cdot13\cdot41$&$2$&$^2$$E_6(2)$&$2^{36}\cdot3^9\cdot5^2\cdot7^2\cdot11\cdot13\cdot17\cdot19$&$6$\\
  $Th$&$2^{15}\cdot3^{10}\cdot5^3\cdot7^2\cdot13\cdot19\cdot41$&$1$&$S_{12}(2)$&$2^{36}\cdot3^8\cdot5^3\cdot7^2\cdot11\cdot13\cdot17\cdot31$&$1$\\
  $A_{20}$&$2^{17}\cdot3^8\cdot5^4\cdot7^2\cdot11\cdot13\cdot17\cdot19$&$2$&$A_{24}$&$2^{21}\cdot3^{10}\cdot5^4\cdot7^3\cdot11^2\cdot13\cdot17\cdot19\cdot23$&$2$\\
  $Fi_{23}$&$2^{18}\cdot3^{13}\cdot5^2\cdot7\cdot11\cdot13\cdot17\cdot23$&$1$&$Fi^{\prime}_{24}$&$2^{21}\cdot3^{16}\cdot5^2\cdot7^3\cdot11\cdot13\cdot17\cdot23\cdot29$&$2$\\
  $Co_1$&$2^{21}\cdot3^{9}\cdot5^4\cdot7^2\cdot11\cdot13\cdot23$&$1$&$A_{25}$&$2^{21}\cdot3^{10}\cdot5^6\cdot7^3\cdot11^2\cdot13\cdot17\cdot19\cdot23$&$2$\\
  $A_{21}$&$2^{17}\cdot3^9\cdot5^4\cdot7^3\cdot11\cdot13\cdot17\cdot19$&$2$&$L_2(13^2)$&$2^3\cdot3\cdot5\cdot7\cdot13^2\cdot17$&$2^2$\\

  $O^{+}_{12}(2)$&$2^{30}\cdot3^{8}\cdot5^2\cdot7^2\cdot11\cdot17\cdot31$&$2$&$L_2(2^{10})$&$2^{10}\cdot3\cdot5^2\cdot11\cdot31\cdot41$&$2\cdot5$\\
  $O^{-}_{12}(2)$&$2^{30}\cdot3^{6}\cdot5^3\cdot7\cdot11\cdot13\cdot17\cdot31$&$2$&$A_{26}$&$2^{22}\cdot3^{10}\cdot5^6\cdot7^3\cdot11^2\cdot13^2\cdot17\cdot19\cdot23$&$2$\\
  $A_{27}$&$2^{22}\cdot3^{13}\cdot5^6\cdot7^3\cdot11^2\cdot13^2\cdot17\cdot19\cdot23$&$2$&$A_{28}$&$2^{24}\cdot3^{13}\cdot5^6\cdot7^4\cdot11^2\cdot13^2\cdot17\cdot19\cdot23$&$2$\\
  $A_{29}$&$2^{24}\cdot3^{13}\cdot5^6\cdot7^4\cdot11^2\cdot13^2\cdot17\cdot19\cdot23\cdot29$&$2$&$A_{30}$&$2^{25}\cdot3^{14}\cdot5^7\cdot7^4\cdot11^2\cdot13^2\cdot17\cdot19\cdot23\cdot29$&$2$\\
  $A_{31}$&$2^{25}\cdot3^{14}\cdot5^7\cdot7^4\cdot11^2\cdot13^2\cdot17\cdot19\cdot23\cdot29\cdot31$&$2$&$A_{32}$&$2^{30}\cdot3^{14}\cdot5^7\cdot7^4\cdot11^2\cdot13^2\cdot17\cdot19\cdot23\cdot29$&$2$\\
  $M$&$2^{46}\cdot3^{20}\cdot5^9\cdot7^6\cdot11^2\cdot13^3\cdot17\cdot19\cdot23\cdot29\cdot31\cdot41\cdot47\cdot59\cdot71$&$1$&$B$&$2^{41}\cdot3^{13}\cdot5^6\cdot7^2\cdot11\cdot13\cdot17\cdot19\cdot23\cdot31\cdot47$&$1$\\
  \hline
 \end{longtabu} }

Since all groups in Table \ref{tb:M} except $L_2(71)$ and  $M$ have their orders and their automorphism group orders not divided by 71, we have  $71\in\pi(H)$. Hence,   $|H_{71}|=71$ and $H_{71}\unlhd G$. Thereby,  we get a contradiction to $t(G)=1$ since $|G/C_G(H_{71})|\big|2\cdot5\cdot7$.

 Suppose  that $K/H\cong M$, then  $G\cong S\cong M$, we are done. Assume that $K/H\cong L_2(71)$, since $|L_2(71)|=2^3\cdot3^2\cdot5\cdot7\cdot71$ and $|\rm{Out}(L_2(71))|=2$, we deduce that  $19\in\pi(H)$. Similarly,  we get a contradiction to $t(G)=1$ for $|G/C_G(H_{19})|\big|2\cdot3^2$.

 \textbf{Step 2} If $S= B$, then $G\cong B$.

In this case,  $|G|=|B|=2^{41}\cdot3^{13}\cdot5^6\cdot7^2\cdot11\cdot13\cdot17\cdot19\cdot23\cdot31\cdot47$. By Lemma \ref{lemT}, $K/H$ is a group in the list of  \cite{ATLAS}, which are groups in Table \ref{tb:B}.

\setlength{\LTleft}{0pt} \setlength{\LTright}{0pt} 
\setlength\LTleft{0in}
\setlength\LTright{+1in plus 2 fill}
\setlength{\tabcolsep}{3pt}
{\tiny
\begin{longtabu}{p{1.5cm}p{3.5cm}c|p{1.5cm}p{4.7cm}c}
\caption{The non-abelian simple groups$K/H$ of orders dividing $|B|$ }~\label{tb:B}\\

\hline

  $K/H$       & $|K/H|$
  & $|\mathrm{Out}(K/H)|$& $K/H$       & $|K/H|$
  & $|\mathrm{Out}(K/H)|$    \\
 \hline
 \endfirsthead
 \multicolumn{4}{l}{Table \ref{tb:B}: (continued) ~} \\
 \hline
 $K/H$       & $|K/H|$
 & $|\mathrm{Out}(K/H)|$ & $K/H$       & $|K/H|$
  & $|\mathrm{Out}(K/H)|$\\
 \hline
 \endhead

\hline
  $A_5$&$2^{2}\cdot3\cdot5$&$2$&$U_5(2)$&$2^{10}\cdot3^5\cdot5\cdot11$&$2$\\
  $L_2(7)$&$2^3\cdot3\cdot7$&$2$&$HS$&$2^{9}\cdot3^2\cdot5^3\cdot7\cdot11$&$2$\\
  $A_6$&$2^{3}\cdot3^2\cdot5$&$4$&$J_3$&$2^{7}\cdot3^5\cdot5\cdot17\cdot19$&$2$\\
  $L_2(8)$&$2^{3}\cdot3^2\cdot7$&$3$&$L_5(4)$&$2^{20}\cdot3^5\cdot5^2\cdot7\cdot11\cdot17\cdot31$&$2^2$\\
  $L_2(11)$&$2^{2}\cdot3\cdot5\cdot11$&$2$&$O_8^{+}(2)$&$2^{12}\cdot3^5\cdot5^2\cdot7$&$2\cdot3$\\ $L_2{(13)}$&$2^2\cdot3\cdot7\cdot13$&$2$&$O_8^{-}(2)$&$2^{12}\cdot3^4\cdot5\cdot7\cdot17$&$2$\\
  $L_2(17)$&$2^4\cdot3^2\cdot17$&$2$&$^3D_4{(2)}$&$2^{12}\cdot3^4\cdot7^2\cdot13$&$3$\\
  $A_7$&$2^{3}\cdot3^2\cdot5\cdot7$&$2$&$A_{12}$&$2^9\cdot3^5\cdot5^2\cdot7\cdot11$&$2$\\
  $L_2(19)$&$2^2\cdot3^2\cdot5\cdot19$&$2$&$M_{24}$&$2^{10}\cdot3^3\cdot5\cdot7\cdot11\cdot23$&$1$\\
  $L_2(16)$&$2^4\cdot3\cdot5\cdot17$&$2^2$&$G_2(4)$&$2^{12}\cdot3^3\cdot5^2\cdot7\cdot13$&$2$\\
  $L_3(3)$&$2^4\cdot3^3\cdot13$&$2$&$M^cL$&$2^{7}\cdot3^6\cdot5^3\cdot7\cdot11$&$2$\\
  $U_3(3)$&$2^{5}\cdot3^3\cdot7$&$2$&$L_3(9)$&$2^7\cdot3^6\cdot5\cdot7\cdot13$&$2^2$\\
  $L_2(23)$&$2^{3}\cdot3\cdot11\cdot23$&$2$&$L_4(4)$&$2^{12}\cdot3^4\cdot5^2\cdot7\cdot17$&$2^2$\\
  $L_2(25)$&$2^3\cdot3\cdot5^2\cdot13$&$2^2$&$U_4(4)$&$2^{12}\cdot3^2\cdot5^3\cdot13\cdot17$&$2^2$\\
  $M_{11}$&$2^4\cdot3^2\cdot5\cdot11$&$1$&$S_4(8)$&$2^{12}\cdot3^4\cdot5\cdot7^2\cdot13$&$2\cdot3$\\
  $L_2(27)$&$2^2\cdot3^3\cdot7\cdot13$&$2\cdot3$&$L_3(16)$&$2^{12}\cdot3^2\cdot5^2\cdot7\cdot13\cdot17$&$2^3\cdot3$\\
  $L_2(29)$&$2^2\cdot3\cdot5\cdot7\cdot29$&$2$&$A_{13}$&$2^{9}\cdot3^5\cdot5^2\cdot7\cdot11\cdot13$&$2$\\
  $L_2(31)$&$2^5\cdot3\cdot5\cdot31$&$2$&$S_6(3)$&$2^{9}\cdot3^9\cdot5\cdot7\cdot7\cdot13$&$2$\\
  $A_8$&$2^{6}\cdot3^2\cdot5\cdot7$&$2$&$O_7(3)$&$2^9\cdot3^9\cdot5\cdot7\cdot13$&$2$\\
  $L_3(4)$&$2^{6}\cdot3^2\cdot5\cdot7$&$2^2\cdot3$&$G_2(5)$&$2^{6}\cdot3^3\cdot5^6\cdot7\cdot31$&$1$\\ $U_{4}(2)$&$2^2\cdot3^4\cdot5$&$2$&$L_4(5)$&$2^{7}\cdot3^2\cdot5^6\cdot13\cdot31$&$8$\\ $Sz(8)$&$2^6\cdot5\cdot7\cdot13$&$3$&$U_6(2)$&$2^{15}\cdot3^6\cdot5\cdot7\cdot11$&$2\cdot3$\\
  $L_2(32)$&$2^5\cdot3\cdot11\cdot31$&$5$&$L_6(2)$&$2^{15}\cdot3^4\cdot5\cdot7^2\cdot31$&$2$\\
  $L_2(47)$&$2^4\cdot3\cdot23\cdot47$&$2$&$A_{14}$&$2^{10}\cdot3^5\cdot5^2\cdot7^2\cdot11\cdot13$&$2$\\
  $U_4(5)$&$2^{5}\cdot3^4\cdot5^4\cdot7\cdot13$&$2^2$&$S_{8}(2)$&$2^{16}\cdot3^5\cdot5^2\cdot7\cdot17$&$1$\\
  $L_{2}(49)$&$2^4\cdot3\cdot5^2\cdot7^2$&$2^2$&$U_3(4)$&$2^6\cdot3\cdot5^2\cdot13$&$2^2$\\
  $M_{12}$&$2^{6}\cdot3^3\cdot5\cdot11$&$2$&$L_{3}(25)$&$2^{7}\cdot3^2\cdot5^6\cdot7\cdot13\cdot31$&$12$\\
  $U_3(5)$&$2^4\cdot3^2\cdot5^3\cdot7$&$6$&$J_1$&$2^3\cdot3\cdot5\cdot7\cdot11\cdot19$&$1$\\
  $Suz$&$2^{13}\cdot3^7\cdot5^2\cdot7\cdot11\cdot13$&$2$&$A_9$&$2^6\cdot3^4\cdot5\cdot7$&$2$\\
  $L_2(64)$&$2^6\cdot3^2\cdot5\cdot7\cdot13$&$2\cdot3$&$Co_3$&$2^{10}\cdot3^7\cdot5^3\cdot7\cdot11\cdot23$&$1$\\
  $A_{15}$&$2^{10}\cdot3^6\cdot5^2\cdot7^6\cdot19\cdot13$&$4$&$L_4(7)$&$2^{9}\cdot3^4\cdot5^2\cdot7\cdot11\cdot13$&$2$\\
  $L_3(5)$&$2^5\cdot3\cdot5^3\cdot31$&$2$&$S_6(4)$&$2^{18}\cdot3^4\cdot5^3\cdot7\cdot13\cdot17$&$2$\\
  $M_{22}$&$2^7\cdot3^2\cdot5\cdot7\cdot11$&$2$&$O_8^{+}(3)$&$2^{12}\cdot3^{12}\cdot5^2\cdot7\cdot13$&$2^3\cdot3$\\
  $J_2$&$2^7\cdot3^3\cdot5^2\cdot7$&$2$&$L_2(125)$&$2^2\cdot3^2\cdot5^3\cdot7\cdot31$&$2\cdot3$\\
  $S_4(4)$&$2^8\cdot3^2\cdot5^2\cdot17$&$2^2$&$A_{16}$&$2^{14}\cdot3^{6}\cdot5^3\cdot7^2\cdot11\cdot13$&$2$\\
  $S_6(2)$&$2^9\cdot3^4\cdot5\cdot7$&$1$&$O_{10}^{+}(2)$&$2^{20}\cdot3^{5}\cdot5^2\cdot7\cdot17\cdot31$&$2$\\
  $A_{10}$&$2^7\cdot3^4\cdot5^2\cdot7$&$2$&$O_{10}^{-}(2)$&$2^{20}\cdot3^{6}\cdot5^2\cdot7\cdot11\cdot17$&$2$\\
  $L_3(7)$&$2^5\cdot3^2\cdot7^3\cdot19$&$2\cdot3$&$G_2(3)$&$2^6\cdot3^6\cdot7\cdot13$&$2$\\
  $U_4(3)$&$2^7\cdot3^6\cdot5\cdot7$&$2^3$&$Co_2$&$2^{18}\cdot3^6\cdot5^3\cdot7\cdot11\cdot23$&$1$\\
  $S_4(5)$&$2^6\cdot3^2\cdot5^4\cdot13$&$2$&$L_5(2)$&$2^{10}\cdot3^2\cdot5\cdot7\cdot31$&$2$\\
  $U_3(8)$&$2^9\cdot3^4\cdot7\cdot19$&$2$&$Fi_{22}$&$2^{17}\cdot3^9\cdot5^2\cdot7\cdot11\cdot13$&$2$\\
  $L_4(3)$&$2^7\cdot3^6\cdot5\cdot13$&$2^2$&$A_{17}$&$2^{14}\cdot3^{6}\cdot5^3\cdot7^2\cdot11\cdot13\cdot17$&$2$\\
  $M_{23}$&$2^7\cdot3^2\cdot5\cdot7\cdot11\cdot23$&$1$&$S_{10}(2)$&$2^{25}\cdot3^6\cdot5^2\cdot7\cdot11\cdot17\cdot31$&$1$\\
  $^2F_4(2)^{\prime}$&$2^{11}\cdot3^3\cdot5^2\cdot13$&$2$&$HN$&$2^{14}\cdot3^6\cdot5^6\cdot7\cdot11\cdot19$&$2$\\
  $A_{11}$&$2^7\cdot3^4\cdot5^2\cdot7\cdot11$&$2$&$A_{18}$&$2^{15}\cdot3^8\cdot5^3\cdot7^2\cdot11\cdot13\cdot17$&$2$\\
  $F_4(2)$&$2^{24}\cdot3^6\cdot5^2\cdot7^2\cdot13\cdot17$&$2$&$L_6(4)$&$2^{30}\cdot3^{6}\cdot5^3\cdot7^2\cdot11\cdot13\cdot17\cdot31$&$12$\\
  $A_{19}$&$2^{15}\cdot3^8\cdot5^3\cdot7^2\cdot11\cdot13\cdot17\cdot19$&$2$&$^2$$E_6(2)$&$2^{36}\cdot3^9\cdot5^2\cdot7^2\cdot11\cdot13\cdot17\cdot19$&$6$\\
  $S_{12}(2)$&$2^{36}\cdot3^8\cdot5^3\cdot7^2\cdot11\cdot13\cdot17\cdot31$&$1$&$A_{20}$&$2^{17}\cdot3^8\cdot5^4\cdot7^2\cdot11\cdot13\cdot17\cdot19$&$2$\\
  $Fi_{23}$&$2^{18}\cdot3^{13}\cdot5^2\cdot7\cdot11\cdot13\cdot17\cdot23$&$1$&$Co_1$&$2^{21}\cdot3^{9}\cdot5^4\cdot7^2\cdot11\cdot13\cdot23$&$1$\\
 $O^{-}_{12}(2)$&$2^{30}\cdot3^{6}\cdot5^3\cdot7\cdot11\cdot13\cdot17\cdot31$&$2$&$O^{+}_{12}(2)$&$2^{30}\cdot3^{8}\cdot5^2\cdot7^2\cdot11\cdot17\cdot31$&$2$\\
 $B$&$2^{41}\cdot3^{13}\cdot5^6\cdot7^2\cdot11\cdot13\cdot17\cdot19\cdot23\cdot31\cdot47$&$1$\\
  \hline
 \end{longtabu}}

 Suppose that $K/H$ is isomorphic to one of the groups in Table \ref{tb:B} except $L_2(47)$ and $B$, then  $47\in\pi(H)$. Clearly,   $|H_{47}|=47$ and $H_{47}\unlhd G$. It follows from $|G/C_G(H_{47})|\big|2\cdot23$ that $\pi(G)\smallsetminus\{23\}\subseteq\pi(C_G(H_{47}))$. And since $\Gamma(G)$ is disconnected, we have $23\not\in\pi(C_G(H_{47}))$. Considering the action of a $23-$element $g$ on $C_G(H_{47})$, because $13\in\pi(C_G(H_{47}))$, in view of  Lemma \ref{ty4} and \ref{ty5}, we deduce that $23\cdot13\in\pi_e(G)$. This implies that $\Gamma(G)$ is connected, a contradiction.

 Suppose  that   $K/H\cong B$, then $G\cong S\cong B$, we are done. If $K/H\cong L_2(47)$. Since $|L_2(47)|=2^4\cdot3\cdot23\cdot47$ and $|\rm{Out}(L_2(47))|=2$, we see that  $13\in\pi(H)$. Similarly,  we get a contradiction to $t(G)=1$ by $|G/C_G(H_{13})|\big|2^2\cdot3$.

\end{proof}

\begin{lemma} \label{le6} Let $G$ be a group and   $S$   one of the groups: $Ly$, $ O^{\prime}N$, $M^{c}L$, $Th$, $HN$, $He$, $Ru$, $HS$, and $Suz$. If $|G|=|S|$ and the prime graph of $G$ is disconnected, then $G\cong S$.\end{lemma}
\begin{proof}
We prove the lemma upon what $S$ is one by one.

\textbf{Step 1} If $S= Ly$, then $G\cong Ly$.

 In this case,  $|G|=|Ly|=2^{8}\cdot3^{7}\cdot5^6\cdot7\cdot11\cdot31\cdot37\cdot67$,
  $K/H$ is one of the groups in Table \ref{tb:Ly} by  \cite{ATLAS} and Lemma \ref{lemT}.

 \begin{table}[H]
\setlength{\abovecaptionskip}{10pt}
\setlength{\belowcaptionskip}{0pt}
\centering
\caption{The non-abelian simple groups $K/H$ of orders dividing $|Ly|$ }~\label{tb:Ly}
\begin{tabular}{p{1.8cm}lc|p{1.8cm}lc}
 \hline
  $K/H$       & $|K/H|$
  & $|\mathrm{Out}(K/H)|$& $K/H$       & $|K/H|$
  & $|\mathrm{Out}(K/H)|$     \\
  \hline
  $A_5$&$2^{2}\cdot3\cdot5$&$2$&$J_2$&$2^7\cdot3^3\cdot5^2\cdot7$&$2$\\
  $L_2(7)$&$2^3\cdot3\cdot7$&$2$&$L_2(125)$&$2^2\cdot3^2\cdot5^3\cdot7\cdot31$&$6$\\
  $A_6$&$2^{3}\cdot3^2\cdot5$&$4$&$A_{10}$&$2^7\cdot3^4\cdot5^2\cdot7$&$2$\\
  $L_2(8)$&$2^{3}\cdot3^2\cdot7$&$3$&$U_4(3)$&$2^7\cdot3^6\cdot5\cdot7$&$2^3$\\
  $L_2(11)$&$2^{2}\cdot3\cdot5\cdot11$&$2$&$A_{11}$&$2^7\cdot3^4\cdot5^2\cdot7\cdot11$&$2$\\ $A_7$&$2^{3}\cdot3^2\cdot5\cdot7$&$2$&$M^cL$&$2^7\cdot3^6\cdot5^3\cdot7\cdot11$&$2$\\ $U_3(3)$&$2^{5}\cdot3^3\cdot7$&$2$&$G_2(5)$&$2^6\cdot3^3\cdot5^6\cdot7\cdot31$&$1$\\ $M_{11}$&$2^4\cdot3^2\cdot5\cdot11$&$2^2$&$L_2(31)$&$2^5\cdot3\cdot5\cdot31$&$2$\\
  $A_8$&$2^6\cdot3^2\cdot5\cdot7$&$2$&$L_3(4)$&$2^6\cdot3^2\cdot5\cdot7$&$12$\\
  $U_4(2)$&$2^6\cdot3^4\cdot5$&$2$&$L_2(32)$&$2^5\cdot3\cdot11\cdot31$&$5$\\
  $M_{12}$&$2^{6}\cdot3^3\cdot5\cdot11$&$2$&$U_3(5)$&$2^4\cdot3^2\cdot5^3\cdot7$&$6$\\
  $A_9$&$2^6\cdot3^4\cdot5\cdot7$&$2$&$L_3(5)$&$2^{5}\cdot3\cdot5^3\cdot31$&$2$\\
  $M_{22}$&$2^7\cdot3^2\cdot5\cdot7\cdot11$&$2$&$Ly$&$2^{8}\cdot3^7\cdot5^6\cdot7\cdot11\cdot31\cdot37\cdot67$&$1$\\
  \hline
  \end{tabular}\end{table}

Suppose that $K/H$ is isomorphic to one of the groups in Table \ref{tb:Ly} except $Ly$, then   $67\in\pi(H)$. Surely,   $|H_{67}|=67$ and $H_{67}\unlhd G$. It follows from $|G/C_G(H_{67})|\big|2\cdot3\cdot11$ that $\pi(G)\smallsetminus\{11\}\subseteq\pi(C_G(H_{67}))$. Since $\Gamma(G)$ is disconnected , we see that $11\not\in\pi(C_G(H_{67}))$. Considering the action of a $11-$element $g$ on $C_G(H_{67})$, because $31\in\pi(C_G(H_{67}))$ and using Lemma \ref{ty4} and \ref{ty5}, we deduce that $11\cdot31\in\pi_e(G)$, which implies that $\Gamma(G)$ is connected, a contradiction. Hence, $K/H\cong Ly$. Therefore, $G\cong S\cong Ly$, as desired.

\textbf{Step 2} If $S= O^{\prime}N$, then $G\cong O^{\prime}N$.

 In this case,  $|G|=|O^{\prime}N|=2^{9}\cdot3^{4}\cdot5\cdot7^3\cdot11\cdot19\cdot31$. Then
  $K/H$ is one of the groups in Table \ref{tb:O'N} by  \cite{ATLAS} and  Lemma \ref{lemT}.

 \begin{table}[H]
\setlength{\abovecaptionskip}{10pt}
\setlength{\belowcaptionskip}{0pt}
\centering
\caption{The non-abelian simple groups K/H of orders dividing $|O^{\prime}N|$~}\label{tb:O'N}
\setlength{\tabcolsep}{4mm}{
\begin{tabular}{lll|lll}
 \hline
 $K/H$       & $|K/H|$
  & $|\mathrm{Out}(K/H)|$& $K/H$       & $|K/H|$
  & $|\mathrm{Out}(K/H)|$     \\
  \hline
  $A_5$& $2^{2}\cdot3\cdot5$ & $2$&$A_8$&$2^{6}\cdot3^2\cdot5\cdot7$&$2$\\
  $L_2(7)$&$2^3\cdot3\cdot7$&$2$&$L_3(4)$&$2^{6}\cdot3^2\cdot5\cdot7$&$2^2\cdot3$\\
  $A_6$& $2^{3}\cdot3^2\cdot5$&$4$&$U_4(2)$&$2^6\cdot3^4\cdot5$&$2$\\
  $L_2(8)$&$2^{3}\cdot3^2\cdot7$&$3$&$L_2(32)$&$2^5\cdot3\cdot11\cdot31$&$5$\\
  $L_2(11)$&$2^{2}\cdot3\cdot5\cdot11$&$2$&$M_{12}$&$2^{6}\cdot3^3\cdot5\cdot11$&$2$\\
  $A_7$&$2^{3}\cdot3^2\cdot5\cdot7$&$2$&$A_9$&$2^6\cdot3^4\cdot5\cdot7$&$2$\\
  $L_2(19)$&$2^2\cdot3^2\cdot5\cdot19$&$2$&$M_{22}$&$2^7\cdot3^2\cdot5\cdot7\cdot11$&$2$\\
  $U_3(3)$&$2^{5}\cdot3^3\cdot7$&$2$&$S_6(2)$&$2^9\cdot3^4\cdot5\cdot7$&$1$\\
  $M_{11}$&$2^4\cdot3^2\cdot5\cdot11$&$1$&$U_3(8)$&$2^9\cdot3^4\cdot7\cdot19$&$18$\\ $L_2(31)$&$2^5\cdot3\cdot5\cdot31$&$2$&$O^{\prime}N$&$2^9\cdot3^4\cdot5\cdot7^3\cdot11\cdot19\cdot31$&$2$\\
  \hline
  \end{tabular}}\end{table}

Suppose that $K/H$ is isomorphic to one of the groups in Table \ref{tb:O'N} except $ O^{\prime}N$, $L_2(31)$ and  $L_2(32)$, then $31\in\pi(H)$. Of course, $|H_{31}|=31$,  $H_{31}\unlhd G$. It follows from $|G/C_G(H_{31})|\big|2\cdot3\cdot5$ that $\pi(G)\smallsetminus\{5\}\subseteq\pi(C_G(H_{31}))$. Further, since $\Gamma(G)$ is disconnected , we have $5\not\in\pi(C_G(H_{31}))$. Now considering the action of a $5-$element $g$ on $C_G(H_{31})$, and using the fact that $19\in\pi(C_G(H_{31}))$.  Lemma \ref{ty4} and \ref{ty5} yield that $5\cdot19\in\pi_e(G)$,  which implies that $\Gamma(G)$ is connected, a contradiction.

 Suppose  that $K/H\cong L_2(31)$ or $L_2(32)$, then  $19\in\pi(H)$ and $|H_{19}|=19$, also $H_{19}\unlhd G$. Hence, we get  contradictions to $t(G)=1$ as $|G/C_G(H_{19})|\big|2\cdot3^2$.

 Now we have $K/H\cong O^{\prime}N$, then  $G\cong S\cong O^{\prime}N$, as desired.

\textbf{ Step 3} If $S= M^{c}L$, then $G\cong M^{c}L$.

By $|G|=|M^{c}L|=2^{7}\cdot3^6\cdot5^3\cdot7\cdot11$ and Lemma \ref{lemT}, we get that $K/H$ is a simple group in the list of  \cite{ATLAS}, which are groups in Table \ref{tb:McL}.

\begin{table}[H]
\setlength{\abovecaptionskip}{10pt}
\setlength{\belowcaptionskip}{0pt}
\centering
\caption{The non-abelian simple groups $K/H$ orders dividing $|M^cL|$~}\label{tb:McL}
\setlength{\tabcolsep}{4mm}{
\begin{tabular}{llc|llc}
 \hline
 $K/H$       & $|K/H|$
 & $|\mathrm{Out}(K/H)|$& $K/H$       & $|K/H|$
 & $|\mathrm{Out}(K/H)|$      \\
  \hline
  $A_5$& $2^{2}\cdot3\cdot5$ & $2$&$U_{4}(2)$&$2^6\cdot3^4\cdot5$&$2$\\
  $L_2(7)$&$2^3\cdot3\cdot7$&$2$&$M_{12}$&$2^{6}\cdot3^3\cdot5\cdot11$&$2$\\
  $A_6$&$2^{3}\cdot3^2\cdot5$&$4$&$U_3(5)$&$2^4\cdot3^2\cdot5^3\cdot7$&$2\cdot3$\\
  $L_2(8)$&$2^{3}\cdot3^2\cdot7$&$3$&$A_9$&$2^6\cdot3^4\cdot5\cdot7$&$2$\\
  $L_2(11)$&$2^{2}\cdot3\cdot5\cdot11$&$2$&$M_{22}$&$2^7\cdot3^2\cdot5\cdot7\cdot11$&$2$\\
  $A_7$&$2^{3}\cdot3^2\cdot5\cdot7$&$2$&$J_2$&$2^7\cdot3^3\cdot5^2\cdot7$&$2$\\
  $U_3(3)$&$2^{5}\cdot3^3\cdot7$&$2$&$A_{10}$&$2^7\cdot3^4\cdot5^2\cdot7$&$2$\\
  $M_{11}$&$2^4\cdot3^2\cdot5\cdot11$&$1$&$U_4(3)$&$2^7\cdot3^6\cdot5\cdot7$&$2^3$\\
  $A_8$&$2^{6}\cdot3^2\cdot5\cdot7$&$2$&$A_{11}$&$2^7\cdot3^4\cdot5^2\cdot7\cdot11$&$2$\\ $L_3(4)$&$2^{6}\cdot3^2\cdot5\cdot7$&$2^2\cdot3$&$M^cL$&$2^7\cdot3^6\cdot5^3\cdot7\cdot11$&$2$\\
  \hline
  \end{tabular}}\end{table}

Suppose that $K/H$ is isomorphic to one of the groups in in Table \ref{tb:O'N} except $ M^{c}L$, $A_{11}$, $M_{22}$, $M_{12}$, $M_{11}$ and $L_{2}(11)$.  We conclude that  $11\in\pi(H)$. Surely,  $|H_{11}|=11$ and  $H_{11}\unlhd G$. Hence, we come to  contradictions to $t(G)=1$ for $|G/C_G(H_{11})|\big|2\cdot5$.

Suppose  that $K/H$ is isomorphic to one of $ M^{c}L$, $A_{11}$, $M_{22}$, $M_{12}$, $M_{11}$ and $L_{2}(11)$.  If $K/H\cong M^{c}L$, then  $G\cong S\cong M^{c}L$, as desired. For the remaining cases, it always follows that  $5\in\pi(H)$. Here $|H_{5}|=5$ or $5^2$ and $H_{5}\unlhd G$. Consequently,  we can deduce  contradictions to $t(G)=1$ by  $|G/C_G(H_{5})|\big||\rm{GL(2,5)}|$ and $|\rm{GL(2,5)}|=2^5\cdot3\cdot5$.

\textbf{Step 4} If $S= Th$, then  $G\cong Th$.

By $|G|=|Th|=2^{15}\cdot3^{10}\cdot5^3\cdot7^2\cdot13\cdot19\cdot31$ and Lemma \ref{lemT},
we get that $K/H$ is in the list of \cite{ATLAS}, which are in Table \ref{tb:Th}.

\setlength{\LTleft}{0pt} \setlength{\LTright}{0pt} 
\setlength\LTleft{0in}
\setlength\LTright{+1in plus 2 fill}
\setlength{\tabcolsep}{3pt}
\begin{longtabu}{p{2.0cm}p{3.0cm}c|p{2.3cm}p{3.9cm}c}
	\caption{The non-abelian simple groups $K/H$ of orders dividing $|Th|$ }~\label{tb:Th}\\
	
	\hline
	
	$K/H$       & $|K/H|$
	& $|\mathrm{Out}(K/H)|$& $K/H$       & $|K/H|$
	& $|\mathrm{Out}(K/H)|$    \\
	\hline
	\endfirsthead
	\multicolumn{4}{l}{Table \ref{tb:Th}: (continued) ~} \\
	\hline
	$K/H$       & $|K/H|$
	& $|\mathrm{Out}(K/H)|$ & $K/H$       & $|K/H|$
	& $|\mathrm{Out}(K/H)|$\\
	\hline
	\endhead
  \hline
  $A_5$& $2^{2}\cdot3\cdot5$&$2$&$L_3(5)$&$2^{5}\cdot3\cdot5^3\cdot31$&$2$\\
  $L_2(7)$&$2^3\cdot3\cdot7$&$2$&$J_2$&$2^7\cdot3^3\cdot5^2\cdot7$&$2$\\
  $A_6$&$2^{3}\cdot3^2\cdot5$&$4$&$L_{2}(125)$&$2^2\cdot3^2\cdot5^3\cdot7\cdot31$&$6$\\
  $L_2(8)$&$2^{3}\cdot3^2\cdot7$&$3$&$A_{10}$&$2^7\cdot3^4\cdot5^2\cdot7$&$2$\\
  $L_2(13)$&$2^{2}\cdot3\cdot7\cdot13$&$2$&$U_4(3)$&$2^7\cdot3^6\cdot5\cdot7$&$2^3$\\
  $A_7$&$2^{3}\cdot3^2\cdot5\cdot7$&$2$&$G_2(3)$&$2^6\cdot3^6\cdot7\cdot13$&$2$\\
  $L_3(3)$&$2^{4}\cdot3^3\cdot13$&$2$&$U_3(8)$&$2^9\cdot3^4\cdot7\cdot19$&$18$\\
  $U_3(3)$&$2^{3}\cdot3\cdot5^2\cdot13$&$2$&$L_4(3)$&$2^7\cdot3^6\cdot5\cdot13$&$2^2$\\
  $L_{2}(25)$&$2^4\cdot3^2\cdot5\cdot13$&$2^2$&$L_5(2)$&$2^{10}\cdot3^2\cdot5\cdot7\cdot31$&$2$\\
  $L_{2}(27)$&$2^2\cdot3^3\cdot7\cdot13$&$6$&$^2$$F_4(2)^{\prime}$&$2^{11}\cdot3^3\cdot5^2\cdot13$&$2$\\
  $A_8$&$2^{6}\cdot3^2\cdot5\cdot7$&$2$&$L_3(9)$&$2^7\cdot3^6\cdot5\cdot7\cdot13$&$2^2$\\ $L_3(4)$&$2^{6}\cdot3^2\cdot5\cdot7$&$2^2\cdot3$&$O^{+}_8(2)$&$2^{12}\cdot3^5\cdot5^2\cdot7$&$6$\\
  $U_{4}(2)$&$2^6\cdot3^4\cdot5$&$2$&$^3$$D_4(2)$&$2^{12}\cdot3^4\cdot7^2\cdot13$&$3$\\
  $Sz(8)$&$2^6\cdot5\cdot7\cdot13$&$3$&$G_2(4)$&$2^{12}\cdot3^3\cdot5^2\cdot7\cdot13$&$2$\\
  $L_{2}(32)$&$2^5\cdot3\cdot11\cdot31$&$5$&$S_4(8)$&$2^{12}\cdot3^4\cdot5\cdot7^2\cdot13$&$6$\\
  $L_{2}(49)$&$2^4\cdot3\cdot5^2\cdot7^2$&$4$&$S_6(3)$&$2^{9}\cdot3^9\cdot5\cdot7\cdot13$&$2$\\
  $U_{3}(4)$&$2^6\cdot3\cdot5^2\cdot13$&$4$&$O_7(3)$&$2^{9}\cdot3^9\cdot5\cdot7\cdot13$&$2$\\
  $U_3(5)$&$2^4\cdot3^2\cdot5^3\cdot7$&$2\cdot3$&$L_6(2)$&$2^{15}\cdot3^4\cdot5\cdot7^2\cdot31$&$2$\\
  $A_9$&$2^6\cdot3^4\cdot5\cdot7$&$2$&$Th$&$2^{15}\cdot3^{10}\cdot5^3\cdot7^2\cdot13\cdot19\cdot31$&$1$\\
  $L_{2}(64)$&$2^6\cdot3^2\cdot5\cdot7\cdot13$&$6$&$L_2(19)$&$2^2\cdot3^2\cdot5\cdot19$&$2$\\
  \hline
  \end{longtabu}

Suppose that $K/H$ is isomorphic to one of the groups in Table \ref{tb:Th} except $ Th$, $L_{2}(19)$ and $U_{3}(8)$, then   $19\in\pi(H)$. Surely, $|H_{19}|=19$ and $H_{19}\unlhd G$.
Suppose  that $K/H$ is one of $ Th$, $L_{2}(19)$ and $U_{3}(8)$.  If $K/H\cong Th$, then  $G\cong S\cong Th$, as desired. For the remain cases, we have $13\in\pi(H)$,  $|H_{13}|=13$ and  $H_{13}\unlhd G$.  Therefore, we can get  contradictions to $t(G)=1$ by  $|G/C_G(H_{13})|\big|2^2\cdot3$ and  $|G/C_G(H_{19})|\big|2\cdot3^2$, respectively.

\textbf{Step 5} If $S= HN$, then $G\cong HN$.

By $|G|=|HN|=2^{14}\cdot3^6\cdot5^6\cdot7\cdot11\cdot19$, Lemma \ref{lemT} and
 \cite{ATLAS}, we have, $K/H$ is one of groups in Table \ref{tb:HN}.

\setlength{\LTleft}{0pt} \setlength{\LTright}{0pt} 
\setlength\LTleft{0in}
\setlength\LTright{+1in plus 2 fill}
\setlength{\tabcolsep}{3pt}
\begin{longtabu}{p{1.5cm}p{3.5cm}c|p{1.5cm}p{4.7cm}c}
\caption{The non-abelian simple groups $K/H$ of orders dividing $|HN|$ }~\label{tb:HN}\\

\hline

  $K/H$       & $|K/H|$
  & $|\mathrm{Out}(K/H)|$& $K/H$       & $|K/H|$
  & $|\mathrm{Out}(K/H)|$    \\
 \hline
 \endfirsthead
 \multicolumn{4}{l}{Table \ref{tb:HN}: (continued) ~} \\
 \hline
 $K/H$       & $|K/H|$
 & $|\mathrm{Out}(K/H)|$ & $K/H$       & $|K/H|$
  & $|\mathrm{Out}(K/H)|$\\
 \hline
 \endhead
  \hline
  $A_5$& $2^{2}\cdot3\cdot5$ & $2$&$A_9$&$2^6\cdot3^4\cdot5\cdot7$&$2$\\
  $L_2(7)$&$2^3\cdot3\cdot7$&$2$&$M_{22}$&$2^7\cdot3^2\cdot5\cdot7\cdot11$&$2$\\
  $A_6$&$2^{3}\cdot3^2\cdot5$&$4$&$J_2$&$2^7\cdot3^3\cdot5^2\cdot7$&$2$\\
  $L_2(8)$&$2^{3}\cdot3^2\cdot7$&$3$&$S_6(2)$&$2^9\cdot3^4\cdot5\cdot7$&$1$\\
  $L_2(11)$&$2^{2}\cdot3\cdot5\cdot11$&$2$&$A_{10}$&$2^7\cdot3^4\cdot5^2\cdot7$&$2$\\
  $A_7$&$2^{3}\cdot3^2\cdot5\cdot7$&$2$&$U_4(3)$&$2^7\cdot3^6\cdot5\cdot7$&$2^3$\\
  $L_2(19)$&$2^{2}\cdot3^2\cdot5\cdot19$&$2$&$U_3(8)$&$2^9\cdot3^4\cdot7\cdot19$&$18$\\
  $U_3(3)$&$2^{5}\cdot3^3\cdot7$&$2$&$U_5(2)$&$2^{10}\cdot3^5\cdot5\cdot11$&$2$\\
  $M_{11}$&$2^4\cdot3^2\cdot5\cdot11$&$1$&$A_{11}$&$2^7\cdot3^4\cdot5^2\cdot7\cdot11$&$2$\\
  $A_8$&$2^{6}\cdot3^2\cdot5\cdot7$&$2$&$HS$&$2^7\cdot3^4\cdot5^2\cdot7\cdot11$&$2$\\ $L_3(4)$&$2^{6}\cdot3^2\cdot5\cdot7$&$2^2\cdot3$&$M^cL$&$2^9\cdot3^2\cdot5^3\cdot7\cdot11$&$2$\\
  $U_{4}(2)$&$2^6\cdot3^4\cdot5$&$2$&$O^{+}_8(2)$&$2^{12}\cdot3^5\cdot5^2\cdot7$&$6$\\
  $M_{12}$&$2^{6}\cdot3^3\cdot5\cdot11$&$2$&$A_{12}$&$2^9\cdot3^5\cdot5^2\cdot7\cdot11$&$2$\\
  $U_3(5)$&$2^4\cdot3^2\cdot5^3\cdot7$&$2\cdot3$&$M^cL$&$2^7\cdot3^6\cdot5^3\cdot7\cdot11$&$2$\\
  $J_1$&$2^3\cdot3\cdot5\cdot7\cdot11\cdot19$&$1$&$HN$&$2^{14}\cdot3^6\cdot5^6\cdot7\cdot11\cdot19$&$2$\\

  \hline
  \end{longtabu}

Suppose that $K/H$ is isomorphic to one of the groups in Table \ref{tb:HN} except $ HN$, $L_{2}(19)$, $J_1$ and $U_{3}(8)$, then   $19\in\pi(H)$. Surely, $|H_{19}|=19$ and $H_{19}\unlhd G$. Further,  we deduce   contradictions to $t(G)=1$ since $|G/C_G(H_{19})|\big|2\cdot3^2$.

Suppose  that   $K/H$ is isomorphic to one of $ HN$, $L_{2}(19)$, $J_1$ and $U_{3}(8)$.  If $K/H\cong HN$, then  $G\cong S\cong HN$, as desired. For the remain cases, we can argue as follows.

 Assume that  $K/H\cong L_{2}(19)$ or $U_{3}(8)$. Since \begin{center} $|L_{2}(19)|=2^2\cdot3^2\cdot5\cdot19$, $|\rm{Out}(L_{2}(19))|=2$;
  $|U_{3}(8)|=2^9\cdot3^4\cdot7\cdot19$, $|\rm{Out}(U_{3}(8))|=2\cdot3^2$.
 \end{center}
 We conclude that $11\in\pi(H)$ and $|H_{11}|=11$, also $H_{11}\unlhd G$. Further,  we deduce   contradictions to  $t(G)=1$ as $|G/C_G(H_{11})|\big|2\cdot5$. Now assume that  $K/H\cong J_1$, notice that $|J_1|=2^3\cdot3\cdot5\cdot7\cdot11\cdot19$, $|\rm{Out}(J_1)|=2$.  We see that $5\in\pi(H)$ and $|H_5|=5^5$,  $H_5\unlhd G$. Since $|G/C_G(H_5)|\big||GL(5,5)|$ and $|GL(5,5)|=2^{13}\cdot3^2\cdot5^{10}\cdot11\cdot13\cdot31\cdot37$, we deduce that $\{2,3,7,19\}\subseteq\pi(C_G(H_5))$. In view of the fact that  $\Gamma(G)$ is disconnected, we see  that $11\not\in\pi(C_G(H_5))$. Considering the action of a $11-$ element $g$ on $C_G(H_5)$, we obtain that $C_G(H_5)$ has a $\left\langle g \right\rangle-$invariant ${Sylow}\ 19-$subgroup, so $11\cdot19\in\pi_e(G)$, which implies that $t(G)=1$, a contradiction.

\textbf{Step 6} If $S= He$, then $G\cong He$.

In this case,  $|G|=|He|=2^{10}\cdot3^3\cdot5^2\cdot7^3\cdot17$.
By \cite{ATLAS} and Lemma \ref{lemT}, we get that $K/H$ is a group in Table \ref{tb:He}.

 \begin{table}[h]
\setlength{\abovecaptionskip}{10pt}
\setlength{\belowcaptionskip}{0pt}
\centering
\caption{The non-abelian simple groups $K/H$ of orders dividing $|He|$~}\label{tb:He}
\setlength{\tabcolsep}{4mm}{
\begin{tabular}{llc|llc}
 \hline
   $K/H$       & $|K/H|$
  & $|\mathrm{Out}(K/H)|$& $K/H$       & $|K/H|$
  & $|\mathrm{Out}(K/H)|$      \\
  \hline
  $A_5$& $2^{2}\cdot3\cdot5$ & $2$&$U_3(3)$&$2^{5}\cdot3^3\cdot7$&$2$\\
  $L_2(7)$&$2^3\cdot3\cdot7$&$2$&$A_8$&$2^{6}\cdot3^2\cdot5\cdot7$&$2$\\
  $A_6$&$2^{3}\cdot3^2\cdot5$&$4$&$L_3(4)$&$2^{6}\cdot3^2\cdot5\cdot7$&$2^2\cdot3$\\
  $L_2(8)$&$2^{3}\cdot3^2\cdot7$&$3$&$L_2(49)$&$2^4\cdot3\cdot5^2\cdot7^2$&$2^2$\\
  $L_2(17)$&$2^{4}\cdot3^2\cdot17$&$2$&$J_2$&$2^7\cdot3^3\cdot5^2\cdot7$&$2$\\
  $A_7$&$2^{3}\cdot3^2\cdot5\cdot7$&$2$&$S_4(4)$&$2^8\cdot3^2\cdot5^2\cdot17$&$2^2$\\
  $L_2(16)$&$2^{4}\cdot3^5\cdot17$&$4$&$He$&$2^{10}\cdot3^3\cdot5^2\cdot7^3\cdot17$&$2$\\
  \hline
  \end{tabular}}\end{table}

Suppose that $K/H$ is isomorphic to one of the groups in Table \ref{tb:He} except $He$, $L_{2}(17)$, $L_{2}(16)$, $S_4(4)$, then  $17\in\pi(H)$. Clearly,   $|H_{17}|=17$ and $H_{17}\unlhd G$. Hence, we get  contradictions to $t(G)=1$ by $|G/C_G(H_{17})|\big|2^4$.

Suppose  that  $K/H\cong He$, $L_{2}(17)$, $L_{2}(16)$, $S_4(4)$.  If $K/H\cong He$, then  $G\cong S\cong He$, as desired. For the remaining cases, we can argue as follows.

  Observe that   $|L_{2}(17)|=2^4\cdot3^2\cdot17$, $|\rm{Out}(L_{2}(17))=2$;
  $|L_{2}(16)|=2^4\cdot3\cdot5\cdot17$, $|\rm{Out}(L_{2}(16))|=4$;
  $|S_{4}(4)|=2^8\cdot3^2\cdot5^2\cdot17$, $|\rm{Out}(S_{4}(4))=4$.
We deduce that $7\in\pi(H)$.  Clearly $|H_{7}|=7^3$  also $H_{7}\unlhd G$. Since $|G/C_G(H_{7})|\big||GL(3,7)|$ and $|GL(3,7)|=2^6\cdot3^4\cdot7^3\cdot19$, we conclude that $\{2,5,17\}\subseteq\pi(C_G(H_{7}))$. Because $t(G)>1$, we have $3\not\in\pi(C_G(H_{7}))$.  Now considering the action of  an element $g$ of order $3$ on $C_G(H_{7})$  and using  Lemma \ref{ty4} and \ref{ty5}, we deduce that $C_G(H_{7})$ has a $\left\langle g \right\rangle-$invariant $Sylow \ 17-$subgroup, therefore, $3\cdot17\in\pi_e(G)$,  which implies that $t(G)=1$, a contradiction.

\textbf{Step 7} If  $S= Ru$, then $S\cong Ru$.

By $|G|=|Ru |=2^{14}\cdot3^{3}\cdot5^3\cdot7\cdot13\cdot29$, Lemma \ref{lemT} and \cite{ATLAS}, we get that $K/H$ is a group in Table \ref{tb:Ru}.

 \begin{table}[H]
\setlength{\abovecaptionskip}{10pt}
\setlength{\belowcaptionskip}{0pt}
\centering
\caption{The non-abelian simple groups $K/H$ of  orders dividing $|Ru|$~}\label{tb:Ru}
\setlength{\tabcolsep}{4mm}{
\begin{tabular}{lll|lll}
 \hline
   $K/H$       & $|K/H|$
  & $|\mathrm{Out}(K/H)|$& $K/H$       & $|K/H|$
  & $|\mathrm{Out}(K/H)|$      \\
  \hline
  $A_5$&$2^{2}\cdot3\cdot5$&$2$&$L_2(29)$&$2^2\cdot3\cdot5\cdot7\cdot29$&$2$\\
  $L_2(7)$&$2^3\cdot3\cdot7$&$2$&$A_8$&$2^{6}\cdot3^2\cdot5\cdot7$&$2$\\
  $A_6$&$2^{3}\cdot3^2\cdot5$&$4$&$L_3(4)$&$2^{6}\cdot3^2\cdot5\cdot7$&$2^2\cdot3$\\
  $L_2(8)$&$2^{3}\cdot3^2\cdot7$&$3$&$S_z(8)$&$2^6\cdot5\cdot7\cdot13$&$3$\\
  $L_2(13)$&$2^{2}\cdot3\cdot7\cdot13$&$2$&$U_{3}(4)$&$2^6\cdot3\cdot5^2\cdot13$&$4$\\
  $A_7$&$2^{3}\cdot3^2\cdot5\cdot7$&$2$&$U_{3}(5)$&$2^4\cdot3^2\cdot5^3\cdot7$&$6$\\
  $L_3(3)$&$2^{4}\cdot3^3\cdot13$&$2$&$L_2(64)$&$2^6\cdot3^2\cdot5\cdot7\cdot13$&$6$\\
  $U_3(3)$&$2^{5}\cdot3^3\cdot7$&$2$&$J_2$&$2^7\cdot3^3\cdot5^2\cdot7$&$2$\\
  $L_2(25)$&$2^3\cdot3\cdot5^2\cdot13$&$4$&$^2$$F_4(2)^{\prime}$&$2^{11}\cdot3^3\cdot5^2\cdot13$&$2$\\
  $L_2(27)$&$2^2\cdot3^3\cdot7\cdot13$&$6$&$G_2(4)$&$2^{12}\cdot3^2\cdot5^2\cdot7\cdot13$&$2$\\
  $Ru$&$2^{14}\cdot3^3\cdot5^3\cdot7\cdot13\cdot29$&$1$\\
  \hline
  \end{tabular}}\end{table}

Suppose that $K/H$ is isomorphic to one of the groups in Table \ref{tb:Ru} except $Ru$ and $L_{2}(29)$, then   $29\in\pi(H)$,  $|H_{29}|=29$ and $H_{29}\unlhd G$. Since $|G/C_G(H_{29})|\big|2^2\cdot7$, it follows that $\{2,3,5,13\}\subseteq\pi(C_G(H_{29}))$. Because $t(G)>1$, we have $7\not\in\pi(C_G(H_{29}))$. Viewing  the action of an element $g$ of order $7$ on $C_G(H_{29})$ and using   Lemma \ref{ty4} and \ref{ty5}, we deduce  that $C_G(H_{29})$ has a $\left\langle g \right\rangle-$invariant $Sylow \ 13-$subgroup. Hence, $7\cdot13\in\pi_e(G)$, so that $t(G)=1$, a contradiction.

Suppose  that    $K/H\cong Ru$ or $L_{2}(29)$.  If $K/H\cong Ru$, then  $G\cong S\cong Ru$, we are done.
Assume that $K/H\cong$$L_{2}(29)$. Because $|L_{2}(29)|=2^2\cdot3\cdot5\cdot7\cdot29$, $|\rm{Out}(L_{2}(29))|=2$,
 we have that $13\in\pi(H)$ and $|H_{13}|=13$, and $H_{13}\unlhd G$. Consequently, we deduce  a contradiction to  $t(G)=1$ from   $|G/C_G(H_{13})|\big|2^2\cdot3$.

\textbf{Step 8} If $S= HS$, then $G\cong HS$.

In this case, $|G|=|HS |=2^{9}\cdot3^2\cdot5^3\cdot7\cdot11$.
Applying Lemma \ref{lemT}, we get  from \cite{ATLAS} that $K/H$ is one of the groups in Table \ref{tb:HS}.

 \begin{table}[H]
\setlength{\abovecaptionskip}{10pt}
\setlength{\belowcaptionskip}{0pt}
\centering
\caption{the non-abelian simple groups $K/H$ of  orders dividing $|HS|$~}\label{tb:HS}
\setlength{\tabcolsep}{4mm}{
\begin{tabular}{lll|lll}
 \hline
   $K/H$       & $|K/H|$
  & $|\mathrm{Out}(K/H)|$& $K/H$       & $|K/H|$
  & $|\mathrm{Out}(K/H)|$      \\
  \hline
  $A_5$& $2^{2}\cdot3\cdot5$ & $2$&$M_{11}$&$2^4\cdot3^2\cdot5\cdot11$&$1$\\
  $L_2(7)$&$2^3\cdot3\cdot7$&$2$&$A_8$&$2^{6}\cdot3^2\cdot5\cdot7$&$2$\\
  $A_6$&$2^{3}\cdot3^2\cdot5$&$4$&$L_3(4)$&$2^{6}\cdot3^2\cdot5\cdot7$&$2^2\cdot3$\\
  $L_2(8)$&$2^{3}\cdot3^2\cdot7$&$3$&$U_3(5)$&$2^4\cdot3^2\cdot5^3\cdot7$&$2\cdot3$\\
  $L_2(11)$&$2^{2}\cdot3\cdot5\cdot11$&$2$&$M_{22}$&$2^7\cdot3^2\cdot5\cdot7\cdot11$&$2$\\
  $A_7$&$2^{3}\cdot3^2\cdot5\cdot7$&$2$&$HS$&$2^9\cdot3^2\cdot5^3\cdot7\cdot11$&$2$\\
  \hline
  \end{tabular}}\end{table}

Suppose that $K/H$ is isomorphic to one of the groups in Table \ref{tb:HS} except $ HS$, $M_{22}$, $M_{11}$ and $L_2(11)$, then    $11\in\pi(H)$ and    $|H_{11}|=11$. Since $H_{11}\unlhd G$, we have $|G/C_G(H_{11})|\big|2\cdot5$, which implies that  $t(G)=1$, a contradiction.

Suppose that $K/H$ is isomorphic to one of $M_{22}$, $M_{11}$  and  $L_{2}(11)$, then $5\in\pi(H)$ and $|H_{5}|=5^2$. Since $H_{5}\unlhd G$, we have   $|G/C_G(H_{5})|\big||GL(2,5)|$. Notice $|GL(2,5)|=2^5\cdot3\cdot5$, we come to $\pi(C_G(H_{5}))=\pi(G)$,  a contradiction to $t(G)=1$.

Therefore  $K/H\cong HS$, then  $G\cong S\cong HS$, as desired.

\textbf{Step 9} If  $S= Suz$, then $G\cong Suz$.

In this case,  $|G|=|Suz|=2^{13}\cdot3^7\cdot5^2\cdot7\cdot11\cdot13$. By the same reason, we get that $K/H$ is isomorphic to a group in Table \ref{tb:Suz}.

\begin{table}[h]
\setlength{\abovecaptionskip}{10pt}
\setlength{\belowcaptionskip}{0pt}
\centering
\caption{The non-abelian simple groups $K/H$ of orders dividing $|Suz|$ }~\label{tb:Suz}
\setlength{\tabcolsep}{4mm}{
\begin{tabular}{lll|lll}
 \hline
   $K/H$       & $|K/H|$
  & $|\mathrm{Out}(K/H)|$& $K/H$       & $|K/H|$
  & $|\mathrm{Out}(K/H)|$      \\
  \hline
  $A_5$&$2^{2}\cdot3\cdot5$&$2$&$M_{22}$&$2^7\cdot3^2\cdot5\cdot7\cdot11$&$2$\\
  $A_6$&$2^{3}\cdot3^2\cdot5$&$2^2$&$U_{5}(2)$&$2^{10}\cdot3^5\cdot5\cdot11$&$2$\\
  $U_4(2)$&$2^6\cdot3^4\cdot5$&$2$&$A_{11}$&$2^7\cdot3^4\cdot5^2\cdot7\cdot11$&$2$\\
  $L_2(7)$&$2^3\cdot3\cdot7$&$2$&$A_{12}$&$2^9\cdot3^5\cdot5^2\cdot7\cdot11$&$2$\\
  $L_2(8)$&$2^{3}\cdot3^2\cdot7$&$3$&$L_2(13)$&$2^2\cdot3\cdot7\cdot13$&$2$\\
  $A_7$&$2^{3}\cdot3^2\cdot5\cdot7$&$2$&$L_3(3)$&$2^{4}\cdot3^3\cdot13$&$2$\\ $U_3(3)$&$2^{5}\cdot3^3\cdot7$&$2$&$L_2(25)$&$2^{3}\cdot3\cdot5^2\cdot13$&$2^2$\\
  $A_8$&$2^6\cdot3^2\cdot5\cdot7$&$2$&$L_2(27)$&$2^2\cdot3^3\cdot7\cdot13$&$2\cdot3$\\
  $L_3(4)$&$2^6\cdot3^2\cdot5\cdot7$&$2^2\cdot3$&$Sz(8)$&$2^6\cdot5\cdot7\cdot13$&$3$\\
  $A_9$&$2^6\cdot3^4\cdot5\cdot7$&$2$&$U_3(4)$&$2^6\cdot3\cdot5^2\cdot13$&$2^2$\\
  $J_2$&$2^7\cdot3^3\cdot5^2\cdot7$&$2$&$L_2(64)$&$2^6\cdot3^2\cdot5\cdot7\cdot13$&$2\cdot3$\\        $S_6(2)$&$2^9\cdot3^4\cdot5\cdot7$&$1$&$G_2(3)$&$2^6\cdot3^6\cdot7\cdot13$&$2$\\
  $A_{10}$&$2^7\cdot3^4\cdot5^2\cdot7$&$2$&$L_4(3)$&$2^{6}\cdot3^6\cdot5\cdot13$&$2^2$\\
  $U_4(3)$&$2^{7}\cdot3^6\cdot5\cdot7$&$2^3$&$^2F_4(2)^{\prime}$&$2^{11}\cdot3^3\cdot5^2\cdot13$&$2$\\
  $O_8^{+}(2)$&$2^{12}\cdot3^5\cdot5^2\cdot7$&$2\cdot3$&$L_3(9)$&$2^7\cdot3^6\cdot5\cdot7\cdot13$&$2^2$\\                                      $L_2(11)$&$2^2\cdot3\cdot5\cdot11$&$2$&$G_2(4)$&$2^{12}\cdot3^3\cdot5^2\cdot7\cdot13$&$2$\\
  $M_{11}$&$2^4\cdot3^2\cdot5\cdot11$&$1$&$A_{13}$&$2^{9}\cdot3^5\cdot5^2\cdot7\cdot11\cdot13$&$2$\\
  $M_{12}$&$2^{6}\cdot3^3\cdot5\cdot11$&$2$&$Suz$&$2^{13}\cdot3^7\cdot5^2\cdot7\cdot11\cdot13$&$2$\\
  \hline
  \end{tabular}}\end{table}

Suppose that $K/H$ is isomorphic to one of the groups in Table \ref{tb:Suz} except $L_2(13)$, $L_{3}(3)$, $L_{2}(25)$, $L_{2}(27)$, $Sz(8)$, $U_3(4)$, $L_2(64)$,  $G_2(3)$, $L_4(3)$, $^2F_4(2)^{\prime}$, $L_{3}(9)$, $G_2(4)$, $A_{13}$ and $Suz$, then $13\in\pi(H)$ , $|H_{13}|=13$ and $H_{13}\unlhd G$. Consequently, we get   contradictions to $t(G)=1$ by  $|G/C_G(H_{13})|\big|2^2\cdot3$.

Suppose that $K/H$ is isomorphic to one of $L_2(13)$, $L_{3}(3)$, $L_{2}(25)$, $L_{2}(27)$, $Sz(8)$, $U_3(4)$, $L_2(64)$,  $G_2(3)$, $L_4(3)$, $^2F_4(2)^{\prime}$, $L_{3}(9)$, $G_2(4)$,  then $11\in\pi(H)$, $|H_{11}|=11$, and $H_{11}\unlhd G$. Hence, $|G/C_G(H_{11})|\big|2\cdot5$, which implies  $t(G)=1$, a  contradictions.

Suppose that $K/H\cong A_{13}$,  then $3\in\pi(H)$ and $|H_{3}|=3^2$. Since $H_{3}\unlhd G$, similarly we get contradictions to $t(G)=1$ from $|G/C_G(H_{3})|\big|2^4\cdot3$.

At last $K/H\cong Suz$, then $G\cong S\cong Suz$, as desired. The lemma holds by Step 1-9.\end{proof}

\subsection{Proof of the Main Theorem}
\begin{proof} The Main Theorem follows from Lemma 2.8-2.13 and  \ref{le1}$-$\ref{le6}.
\end{proof}

\end{document}